\documentclass[12pt]{amsart}
\usepackage[utf8]{inputenc}
\usepackage[utf8]{inputenc}
\usepackage[T1]{fontenc}
\usepackage[all]{xy}
\usepackage{lipsum}
\usepackage{url}
\usepackage{tikz}
\usepackage{stackrel}
\usepackage{color}
\def\red{}
\def\azulito{}
\def\azul{}
\def\blue{}
\usepackage{amsmath,amsthm,amssymb}

\newcommand{\aspas}[1]{``{#1}''}

\usepackage{xcolor}

\usepackage{graphicx}

\theoremstyle{definition}

\newtheorem{theorem}{Theorem}[section]
\newtheorem{lemma}[theorem]{Lemma}
\newtheorem{example}[theorem]{Example}

\newtheorem{proposition}[theorem]{Proposition}
\newtheorem{definition}[theorem]{Definition}

\newtheorem{remark}[theorem]{Remark}
\newtheorem{corollary}[theorem]{Corollary}
\numberwithin{equation}{section}

\begin{document}

\renewcommand{\bf}{\bfseries}
\renewcommand{\sc}{\scshape}

\newcommand{\TC}{\text{TC}}

\title[Higher topological complexity]%
{Higher topological complexity of a map}

\author{Cesar A. Ipanaque Zapata}
\address{Departamento de Matem\'{a}tica, Universidade de S\~{a}o Paulo Instituto de Matem\'{a}tica e  Estatística -IME/USP, R. do Matão, 1010 - Butantã, CEP:
05508-090 - S\~{a}o Paulo, Brasil}
\email{cesarzapata@usp.br}
\thanks{The first author would like to thank grant\#2016/18714-8 and grant\#2022/03270-8, S\~{a}o Paulo Research Foundation (FAPESP) for financial support.}

\author{Jes\'{u}s Gonz\'{a}lez}
\address{Departamento de Matem\'{a}ticas, Centro de Investigaci\'{o}n y de Estudios Avanzados del I. P. N.
Av. Instituto Polit\'{e}cnico Nacional n\'{u}mero 2508,
San Pedro Zacatenco, Mexico City 07000, M\'{e}xico}
\email{jesus@math.cinvestav.mx}



\begin{abstract}
The higher topological complexity of a space $X$, $\text{TC}_r(X)$, $r=2,3,\ldots$, and the topological complexity of a map $f$, $\text{TC}(f)$, have been introduced by Rudyak and Pave\v{s}i\'{c}, respectively, as natural extensions of Farber's topological complexity of a space. In this paper we introduce a notion of higher topological complexity of a map~$f$, $\text{TC}_{r,s}(f)$, for $1\leq s\leq r\geq2$, which simultaneously extends Rudyak's and Pave\v{s}i\'{c}'s notions. Our unified concept is relevant in the $r$-multitasking motion planning problem associated to a robot devise when the forward kinematics map plays a role in $s$ prescribed stages of the motion task. We study the homotopy invariance and the behavior of $\text{TC}_{r,s}$ under products and compositions of maps, as well as the dependence of $\text{TC}_{r,s}$ on $r$ and $s$. We draw general estimates for $\text{TC}_{r,s}(f\colon X\to Y)$ in terms of categorical invariants associated to $X$, $Y$ and $f$. In particular, we describe within one the value of $\text{TC}_{r,s}$ in the case of the non-trivial double covering over real projective spaces, as well as for their complex counterparts. 
\end{abstract}

\maketitle


\section{Introduction}\label{secintro}
In this article \aspas{space} means a topological space, and by a \aspas{map} we will always mean a continuous map. Fibrations are taken in the Hurewicz sense.

\smallskip
\azul{Consider an autonomous robot devise $\mathcal{A}$ performing on a known work space $\mathcal{W}$. The fundamental problem in geometric motion planning~(\cite{MP}) is to find a suitable (safe, efficient, optimal) path taking $\mathcal{A}$ from a given initial configuration to a goal configuration. Here, the term \emph{configuration} refers to a complete specification of every parameter in the robot's geometry at allowable (collision-free) states. If $\mathcal{C}$ stands for the space of all possible configurations of $\mathcal{A}$, the robot operation usually comes in the form of a \emph{forward kinematics map} $F\colon\mathcal{C}\to\mathcal{W}$ where, for a configuration $q\in\mathcal{C}$, $F(q)$ encodes the corresponding effect of the robot in the work space.}

\smallskip
\azul{In practice, motion tasks may involve constraints both on $q$ and on $F(q)$. In such a context, we are interested in a hybrid multitasking version of the motion planning problem. Given a reference configuration $q_0$ of $\mathcal{A}$ and a tuple
\begin{equation}\label{tuple}
(q_1,q_2,\ldots,q_s,e_1,e_2,\ldots,e_\ell)\in\mathcal{C}^s\times\mathcal{W}^{\ell}
\end{equation}
of $s$ desired configurations and $\ell$ desired effects of the robot, the goal is to describe a \emph{solving $r$-multi\-plath~$\gamma$}, namely, a family of paths $\gamma_1,\gamma_2,\ldots,\gamma_r$ in $\mathcal{C}$ with $r=s+\ell$, all starting at~$q_0$, such that:
\begin{itemize}
\item for $1\leq i\leq s$, $\gamma_i$ ends at $q_i$;
\item for $1\leq j\leq \ell$, $\gamma_{s+j}$ ends at a point in the inverse image $F^{-1}(\{e_j\})$. 
\end{itemize}
The model we propose in Section~\ref{htcf} is intended to study the topological instabilities in the resulting motion planning problem. This is done through the introduction of a numerical invariant that measures the minimal number of robust-to-noise instructions needed to solve, in a global manner, the $r$-multitasking motion problem above.}

\smallskip
\azul{Our work is motivated by Rudyak's $r$-th sequential topological complexity $\TC_r(X)$ of a space $X$, developed in~\cite{BGRT,rudyak2010higher}, and by Pave\v{s}i\'{c}'s topological complexity $\TC(f)$ of a map $f\colon X\to Y$, developed in~\cite{pavesic2019}. We define the $(r,s)$-higher topological complexity $\TC_{r,s}(f)$ of $f$ for integers $r\geq 2$ and $1\leq s\leq r$. Here, the parameter $s$ stands for the number of tasks for which the forward kinematic map must be taken into account, while $\ell:=r-s$ is the number of configurations in~(\ref{tuple}) above. Rudyak's and Pave\v{s}i\'{c}'s invariants are recovered with $f=1_X$ and $(r,s)=(2,1)$, respectively.}

\smallskip
\azul{In addition to its relevance in the multitasking problem for the forward kinematics map, the parameter $s$ in our invariant $\TC_{r,s}(f)$ plays a subtle role within more theoretical issues. For starters, our invariant is sensitive to the numbers $r-s$ (of configurations) and~$s$ (of effect tasks), a fact reflected in part by the fairly regular monotonic behavior}
$$\text{TC}_{r,s}(f)\leq\min\{\text{TC}_{r+1,s}(f),\text{TC}_{r+1,s+1}(f)\}$$ (see Proposition~\ref{prop:ineq}). \azul{For instance, for the double cover\azul{ing map} $p_n\colon S^n\to\mathbb{R}P^n$, we show $$\TC_{r,s}(p_n)=r+s(n-1)+\varepsilon_{r,s,n},$$ where\footnote{\azul{The precise value of $\varepsilon_{r,s,n}$ is given in Section~\ref{sec:example} for $r=s$ (any $n$), and for $n\in\{1,3,7\}$ (any $r$ and~$s$).}} $\varepsilon_{r,s,n}\in\{0,1\}$. On the other hand, the well known fact that Rudyak's $\TC_r(Y)$ of an $H$-space $Y$ agrees with the Lusternik-Schnirelmann category of $Y^{r-1}$ is encoded by $\TC_{r,r-1}(f)$ for any fibration over $Y$ (Corollary~\ref{h-space}). In general, the use of the bi-parameter $(r,s)$ allows us to get a discrimination of the topological properties of a space $Y$ in a manner which is finer than that provided by the several higher topological complexities $\TC_r(Y)$. For instance, Rudyak's monotonic behavior $\TC_r(Y)\leq\TC_{r+1}(Y)$ is refined by the inequalities
$$
\TC_r(Y)\leq\TC_{r,r}(f)\leq\TC_{r+1,r}(f)\leq\TC_{r+1}(Y),
$$
valid for any fibration $f\colon X\to Y$ (Remark~\ref{monorefinado}).}

\smallskip
We provide estimates for $\TC_{r,s}(f)$ for a general map $f\colon X\to Y$, \azul{possibly failing to be a fibration.} As a way of illustration, Propositions~\ref{prop:lower-cat-sec} and~~\ref{theorem:cohomological-estimate} \azul{yield}
$$
\max\{ \text{sec}^{\azul{1\times f^s}}(e_r^X),\text{sec}\hspace{.1mm}(f^s),\text{nil}\left(\text{Ker}((\Delta_{r-s},{}^sf)^\ast)\right)\} \leq \text{TC}_{r,s}(f)
\leq \text{sec}\hspace{.1mm}(f^s)\cdot \text{sec}^{\azul{1\times f^s}}(e_r^X).
$$
\azul{We also study the homotopy invariance of $\TC_{r,s}$ together with its behavior under composition of maps. In fact, virtually all properties developed in~\cite{pavesic2019} for the case $(r,s)=(2,1)$ are extended here to the higher TC realm. Yet, unlike Pave{\v{s}}i{\'c}'s approach, we work with the standard (and better suited for actual applications) definition of the sectional number of a map in terms of open coverings (reviewed in Section~\ref{secprelim}).} 

\smallskip
Rudyak and Soumen \cite{rudyak2022} have recently introduced a notion of higher topological complexity $\mathrm{TC}^{RS}_r(f)$ of a map $f$. Their concept is compared to ours. For instance, \red{in Corollary~\ref{cor:unification}, we \azulito{obtain} that the \azulito{equalities} $\mathrm{TC}^{RS}_{r}(f)=\mathrm{TC}_{r,r}(f)=\text{TC}_{r,r-1}(f)$ \azulito{hold} for any fibration $f$ \azulito{admiting} a section. \azulito{Additionally, we show that, for any map $f$ (possibly not a fibration), Rudyak and Soumen's $\mathrm{TC}^{RS}_{r}(f)$ is in fact a generalization of Murillo-Wu's notion of topological complexity of $f$ (Proposition~\ref{rs-wm-s}), and that, under special conditions, our $\TC_{r,s}(f)$ with large $s$ unifies previous notions of topological complexity (Corollary~\ref{cor:unification}).}} 

\section{Preliminaries on sectional numbers}\label{secprelim}
Given a map $f:X\to Y$ and a subset $A$ of $Y$, we say that a map $s:A\to X$ is a \textit{\azul{local} section} \azul{of $f$} if $f\circ s=incl_A$, and a \textit{\azul{local} homotopy section} \azul{of $f$} if  $f\circ s\simeq incl_A$, where $incl_A:A\to Y$ is the inclusion map. \azul{The} \textit{sectional number} $\text{sec}(f)$ is the \azul{least integer~$m$ such that $Y$ can be covered by $m$ open subsets each of which admits a local section of~$f$. We set $\text{sec}(f)=\infty$ if no such $m$ exists. Likewise, the} \textit{sectional category} $\text{secat}(f)$ is the \azul{least integer~$m$ such that $Y$ can be covered by $m$ open subsets each of which admits a local homotopy section of~$f$. Again, we set $\text{secat}(f)=\infty$ if no such $m$ exists.} \azul{See~\cite{berstein1961}.} 

\smallskip
\azul{Note that $f$ is forced to be surjective whenever $\text{sec}(f)<\infty$. Furthermore, the inequality $\text{secat}(f)\leq\text{sec}(f)$ holds for any map $f$. Additionally, from} the homotopy lifting property, a homotopy section of a fibration can \azul{be replaced} by a strict section. \azul{In particular,} $\text{secat}(f)=\text{sec}(f)$ when $f$ is a fibration.

\smallskip
For $f:X\to Y$ and $g:Y\to Z$, we define the \azul{\emph{sectional number}} $\text{sec}^g(f)$ as the least integer $n$ for which $Y$ admits a covering \azul{by} $n$ open sets $U_i$ such that over each~$U_i$ there is a map $s_i:U_i\to X$ with $g\circ f\circ s_i=g_{\mid U_i}$. Likewise, the \textit{sectional category} $\mathrm{secat}^{g}(f)$ is the least integer $n$ for which $Y$ admits a covering by $n$ open sets $U_i$ such that over each~$U_i$ there is a map $s_i:U_i\to X$ with $g\circ f\circ s_i\simeq g_{\mid U_i}$. \azulito{As reviewed at the end of this section,} the invariant $\mathrm{secat}^{g}(f)$ is studied by Murillo and Wu in \cite{murillo2019}. \azul{The following fact is \azulito{straightforward to prove}:} 

\begin{lemma}\label{lem:sec-composite}
Let $f:X\to Y$, $g:Y\to Z$ \azul{and $\phi:Z\to W$} be \azul{arbitrary} maps. We have
$$ \text{sec}^{\phi\circ g}(f)\leq \text{sec}^g(f)\leq\text{sec}(f).$$
\end{lemma}

\azul{Recall} the pathspace construction from \cite[pg. 407]{hatcheralgebraic}. For a \azul{map} $f:X\to Y$, consider the space 
\begin{equation*}
E_f=\{(x,\gamma)\in X\times PY\mid~\gamma(0)=f(x)\},  
\end{equation*} where $PY=Y^I$ is the space of all paths $[0,1]\to Y$. The map \begin{equation*}
\rho_f:E_f\to Y,~(x,\gamma)\mapsto \rho_f(x,\gamma)=\gamma(1),
\end{equation*} is a fibration. Further\azul{more}, the projection \azul{onto} the first coordinate $E_f\to X,~(x,\gamma)\mapsto x$, is a homotopy equivalence with homotopy inverse $c:X\to E_f$ given by $x\mapsto (x,\gamma_{f(x)})$, where $\gamma_{f(x)}$ is the constant path at $f(x)$. \azul{This renders the factorization $$\left(X\stackrel{f}\to Y\right)\,=\,\left(X\stackrel{c}{\to} E_f\stackrel{\rho_f}{\to} Y\right),$$ a composition} of a homotopy equivalence \azul{followed by} a fibration. \azul{Furthermore, $f$ is a fibration if and only $f$ admits a lifting function, i.e., a map $\Gamma\colon E_f\to PX$ such that, for each $(x,\gamma)\in E_f$, we have}
\begin{equation}\label{lifun}
\azul{\Gamma(x,\gamma)(0)=x\text{ \ \ and \ \ } f\circ\Gamma(x,\gamma)=\gamma.}
\end{equation}

\smallskip By a \textit{quasi pullback} we mean a \azul{strictly} commutative diagram
\begin{eqnarray}\label{xfy}
\xymatrix{ \rule{3mm}{0mm}& X^\prime \ar[r]^{\azul{\varphi'}} \ar[d]_{f^\prime} & X \ar[d]^{f} & \\ &
       Y^\prime  \ar[r]_{\,\,\varphi} &  Y &}
\end{eqnarray} 
such that\azul{,} for any \azul{strictly commutative diagram as the one on the left hand-side of~(\ref{diagramadoble}),} there exists a (not necessarily unique) map $h:Z\to X^\prime$ \azul{that renders a strictly commutative diagram as the one on the right hand-side of~(\ref{diagramadoble}).}
\begin{eqnarray}\label{diagramadoble}
\xymatrix{
Z \ar@/_10pt/[dr]_{\alpha} \ar@/^30pt/[rr]^{\beta} & & X \ar[d]^{f}  & & &
Z\rule{-1mm}{0mm} \ar@/_10pt/[dr]_{\alpha} \ar@/^30pt/[rr]^{\beta}\ar[r]^{h} & 
X^\prime \ar[r]^{\azul{\varphi'}} \ar[d]_{f^\prime} & X \\
& Y^\prime  \ar[r]_{\,\,\varphi} &  Y & & & & Y^\prime &  \rule{3mm}{0mm}}
\end{eqnarray}
Note that such a condition amounts to saying that $X'$ contains the canonical pullback $Y'\times_Y X$ determined by $f$ and $\varphi$ as a retract in a way that is compatible with the mappings into $X$ and $Y'$.

\smallskip For convenience, we record the following standard properties, most of which appear in Chapter~4 of~\cite{zapata2022}: 
\begin{lemma}\label{lem:proper-sec-secat}
\noindent 
\begin{enumerate}
\item If~(\ref{xfy}) is a quasi pullback, then $$\text{sec}\hspace{.1mm}(f^\prime)\leq \text{sec}\hspace{.1mm}(f).$$
    \item For a map $f:X\to Y$, $$\text{secat}(\rho_f)= \text{secat}(f).$$
    \item If $f,g:X\to Y$ are homotopic maps (which we shall denote by $f\simeq g$), then $$\text{secat}(f)=\text{secat}(g).$$
    \item If $f:X\to Y$ and $g:Y\to Z$ are maps, then \[\text{sec}(g)\cdot\text{sec}^g(f)\geq\text{sec}(g\circ f)\geq\max\{\text{sec}(g),\text{sec}^g(f)\}.\] In particular, $\text{sec}(g\circ f)=\text{sec}^g(f)$ provided $g$ admits a section.
    \item If $p:E\to B$ is a fibration, then $$\text{sec}\hspace{.1mm}(p)\leq \text{cat}(B).$$ In particular, $\text{secat}(f)\leq \text{cat}(Y)$ for any map $f:X\to Y$.
     \item If $f:X\to Y$ is null-homotopic, then $$\text{secat}(f)=\text{cat}(Y).$$
     \item (cf.~\cite[Proposition 20, pg.~83]{schwarz1966}) Let $f:X\to Y$ be a map with $Y$ normal. If $\{C_1,\ldots,C_k\}$ and $\{D_1,\ldots,D_{\ell}\}$ are open coverings of $Y$ such that on each $C_i\cap D_j$ there exists a section of $f$, then
\[\text{sec}(f)\leq k+\ell-1.\]
\item For a space $Z$ and a map $f:X\to Y$, $$
    \mathrm{sec}\hspace{.1mm}(1_Z\times f)=\mathrm{sec}\hspace{.1mm}(f)\text{ \ \ and \ \ }
    \mathrm{secat}\hspace{.1mm}(1_Z\times f)=\mathrm{secat}\hspace{.1mm}(f).
$$
\end{enumerate}
\end{lemma}

\smallskip
The sectional number of the canonical pullback $\varphi^\ast(p):K\times_B E\to K$ on the left hand-side of~(\ref{comsqre}) below,
denoted by $\text{sec}_\varphi(p)$, \azulito{is} called \textit{relative sectional number}.
\begin{eqnarray}\label{comsqre}
\xymatrix{
K\times_B E \ar[r]^{ } \ar[d]_{\varphi^\ast (p)} & E \ar[d]^{p} & & & 
X \ar[r]^{\phi} \ar[d]_{f} & W \ar[d]^{h} \\
       K \ar[r]_{\,\, \varphi}  &  B & & & Y \ar[r]_{g}  &  Z}
\end{eqnarray}

\begin{lemma}\label{sec-sec}
\azul{The inequalities $\text{sec}_g(h)\leq \text{sec}^g(f)\leq \text{sec}(f)$ hold for any commutative square as the one on the right hand-side of~(\ref{comsqre}). If the square is a quasi pullback, then in fact $\text{sec}_g(h)= \text{sec}^g(f)=\text{sec}(f)$.}
\end{lemma}
\begin{proof}
\azul{For} $U\subset Y$ and $s:U\to X$ satisfying $g\circ f\circ s= g_{\mid U}$, the map $\sigma:U\to W$ given by $\sigma=\phi\circ s$ defines a lift of \azul{$g_{\mid U}$} through $h$. \azul{The first inequality asserted in the lemma then follows by observing (see~\cite[Proposição 4.5.16]{zapata2022}) that $\text{sec}_\varphi(p)$ can be defined in terms} of open covers $\{U_i\}$ of ~\azul{$K$ such} that each element of the cover admits a lift $\sigma_i:U_i\to E$ \azul{of $\varphi_{\mid U_i}$ through~$p$,} i.e., $p\circ \sigma_i=\varphi_{\mid U_i}$. \azul{The second inequality in the lemma comes from Lemma~\ref{lem:sec-composite}. The proof is complete by noticing that $\text{sec}(f)\leq\text{sec}_g(h)$ when the given square is a quasi pullback. Indeed, the quasi pullback hypothesis implies that any lift $\sigma\colon U\to W$ of $g_{\mid U}$ through $h$ can be lifted through $\phi$ to a local section $U\to X$ of $f$.}
\end{proof}

\begin{remark}
\azulito{Note that, when $p$ is a fibration,} $\text{sec}_\varphi(p)$ can be defined in terms of open covers $\{U_i\}$ of $K$ such that each element of the cover admits a homotopic lift $\sigma_i:U_i\to E$ of $\varphi_{\mid U_i}$ through~$p$, i.e., $p\circ \sigma_i\simeq \varphi_{\mid U_i}$.
\end{remark}

\azul{We close the section by \azulito{indicating how the sectional numbers we have just formalized capture the different versions in the literature} of (topological) complexity of \azulito{a map $f\colon X\to Y$}. Let} $e_2^{ X}:PX\to X\times X$ be the \azulito{double-evaluation} fibration given by $e_2^X(\gamma)=(\gamma(0),\gamma(1))$.

\begin{itemize}
\item The \textit{\azul{complexity}} of $f$, $\text{cx}(f)$, introduced by Pave{\v{s}}i{\'c} in~\cite{pavesic2017} \azul{(see also~\cite{pavesic2018}),} is the sectional number
$$\text{sec}(PX\stackrel{e_2^{ X}}{-\!\!\!-\!\!\!\longrightarrow} X\times X\stackrel{1_X\times f}{-\!\!\!-\!\!\!\longrightarrow} X\times Y).$$  \azul{When} $f$ is a fibration between ANRs \azul{spaces, the} number $\text{\azul{cx}}(f)$ coincides with the notion of topological complexity \azul{$\TC(f)$} studied in \cite{pavesic2019}. \azul{The complexity $\text{cx}(f)$ has recently been used in \cite{cesarmilnor,cesarsectional}.} 

\item
A different approach was \azul{taken} by Murillo and Wu in \cite{murillo2019}. \azul{Their} topological complexity of $f$, which we denote by $\mathrm{TC}^{MW}(f)$, is \red{given by 
$$\mathrm{TC}^{MW}(f)=\mathrm{secat}^{f\times f}(e_2^X),$$ i.e., } \azulito{the} least integer $n$ such that $X\times X$ can be covered by $n$ open sets $\{U_i\}_{i=1}^{n}$ on each of which there is a map $s_i:U_i\to PX$ satisfying $(f\times f)\circ e_2^{ X}\circ s_i\simeq (f\times f)_{\mid U_i}$. \azulito{Their} \textit{naive or strict topological complexity} of $f$, which we denote by $\mathrm{tc}^{MW}(f)$, is defined analogously, \azulito{except that one now requires} each of the maps $s_i:U_i\to PX$ to satisfy the stronger condition $(f\times f)\circ e_2^{ X}\circ s_i= (f\times f)_{\mid U_i}$\azul{. In other words,} $$\mathrm{tc}^{MW}(f)=\text{sec}^{f\times f}(e_2^{ X}).$$ As shown in \cite{murillo2019}, the inequality $\mathrm{TC}^{MW}(f)\leq \mathrm{tc}^{MW}(f)$ holds for any map $f$\azul{, while in fact $\mathrm{TC}^{MW}(f)=\mathrm{tc}^{MW}(f)$} when $f$ is a fibration. 

\item As detailed in Subsection~\ref{relationtoRS}, relative sectional numbers are closely related to Rudyak-Soumen's quasi-strong sectional category of a map. In fact, by extending ideas in Scott's study of the relative sectional number $$\text{sec}_{f\times f}(e^Y_2)$$ (\cite[Definition 3.1]{scott2022}), we show that Rudyak-Soumen's higher TC is in fact a generalization of Murillo-Wu's TC of a map. See Proposition~\ref{rs-wm-s} below.
\end{itemize}

\section{\azul{Higher} topological complexity}\label{htcf}
For $r\geq 2$, let $J_r$ be the wedge of $r$ closed intervals $[0,1]_i$, $i=1,\ldots,r$, where the zero points $0_i\in [0,1]_i$ are identified. For a space $X$, let $X^{J_r}$ denote the space of maps $\gamma:J_r\to X$ with the compact-open topology. Consider the fibration\footnote{Since $PX$ is homeomorphic to $X^{J_2}$, the notation $e_r^X$ \azul{is compatible with the use of $e^X_2$ in the previous section}.}
\begin{equation}\label{evaluation-fibration}
    e_r^X:X^{J_r}\to X^r,~e_r(\gamma)=\left(\gamma(1_1),\ldots,\gamma(1_r)\right),
\end{equation}
where \azul{$1_i\in [0,1]_i$. \azulito{Here we} regard} $X^{J_r}$ as the space of ordered \azul{$r$-multipaths} in $X$ all \azulito{whose components} have a common starting point. From \cite{rudyak2010higher}, the \textit{$r$-th \azul{higher} topological complexity} TC$_r(X)$ of $X$ is the sectional number of the fibration~(\ref{evaluation-fibration}). In  other  words, the $r$-th \azul{higher} topological complexity of $X$ is the smallest positive integer TC$_r(X)=k$ for which  the product $X^r$ is covered by $k$ open subsets $X^r=U_1\cup\cdots\cup U_k$ such \azul{that,} for any $i=1,2,\ldots,k$\azul{,} there exists a local section $s_i:U_i\to X^{J_r}$ of $e_r^X$ over $U_i$ (i.e., $e_r^X\circ s_i=incl_{U_i}$).

\smallskip Let $f:X\to Y$ be a map, and let \[e_{r,s}^f:X^{J_r}\to X^{r-s}\times Y^s,~e_{r,s}^f=(1_{X^{r-s}}\times f^s)\circ e_r^X,\] for $1\leq s\leq r$. For example, $e_{r,r-1}^f=(1_{X}\times f^{r-1})\circ e_r^X$ and $e_{r,r}^f=f^r\circ e_r^X$.

\begin{definition}\label{strong-higher-tc}
\noindent \begin{enumerate}
\item The \textit{strong $(r,s)$-th \azul{higher} topological complexity} of \azulito{a} map \azulito{$f\colon X\to Y$,} denoted by TC$_{r,s}(f)$, is the sectional number $\text{sec}\hspace{.1mm}(e_{r,s}^f)$ of the map $e_{r,s}^f$, that is, the least integer $m$ such that the cartesian product $X^{r-s}\times Y^s$ can be covered \azul{by} $m$ open subsets $U_i$ such \azul{that,} for any $i = 1, 2, \ldots , m$\azul{,} there exists a local section $s_i : U_i \to X^{J_r}$ of $e_{r,s}^f$, \azul{so} $e_{r,s}^f\circ s_i = incl_{U_i}$. If no such $m$ exists we set TC$_{r,s}(f)=\infty$. 
\item The \textit{homotopy $(r,s)$-th \azul{higher} topological complexity} of the map $f$, denoted by HTC$_{r,s}(f)$, is the sectional category $\text{secat}\hspace{.1mm}(e_{r,s}^f)$ of the map $e_{r,s}^f$, that is, the least integer $m$ such that the cartesian product $X^{r-s}\times Y^s$ can be covered with $m$ open subsets $U_i$ such \azul{that,} for any $i = 1, 2, \ldots , m$\azul{,} there exists a local homotopy section $s_i : U_i \to X^{J_r}$ of $e_{r,s}^f$, \azul{so} $e_{r,s}^f\circ s_i \simeq incl_{U_i}$. If no such $m$ exists we set HTC$_{r,s}(f)=\infty$.
\end{enumerate}
\end{definition}

\azul{Note that $f$ is forced to be surjective whenever $\TC_{r,s}(f)<\infty$.} The strong form of the \azul{higher} TC of a map is \azul{best suited for} applications. Accordingly, \azul{$\TC_{r,s}(f)$} \azul{will be the main focus in this work.}

\begin{remark}\label{remark:tc-path}
For $r\geq 2$, consider the evaluation fibration ${e_r^\prime}:PX\to X^r$ given by     \[e_r^\prime(\gamma)=\left(\gamma(0),\gamma\left(\dfrac{1}{r-1}\right),\ldots,\gamma\left(\dfrac{r-2}{r-1}\right),\gamma(1)\right).
\] \azul{We have commutative diagrams}
\begin{eqnarray*}
\xymatrix{ PX \ar[rr]^{\phi} \ar[rd]_{e^\prime_r} & & X^{J_r} \ar[ld]^{e_r^X} & \\
        &  X^r & &} & \xymatrix{ X^{J_r} \ar[rd]_{e_r^X}  \ar[rr]^{\psi}  & & PX \ar[ld]^{e^\prime_r}& \\
        &  X^r\azul{,} & &}
\end{eqnarray*}
where \azul{$\phi(\gamma)=(\gamma_1,\ldots,\gamma_r)$ and $\psi(\alpha_1,\ldots,\alpha_r)=\alpha_1\cdot (\overline{\alpha}_1\cdot \alpha_2)\cdot (\overline{\alpha}_2\cdot \alpha_3)\cdots (\overline{\alpha}_{r-1}\cdot \alpha_r).$ Here \azul{$\alpha\cdot\beta$ stands for the} concatenation \azul{of $\alpha$ and $\beta$,} $\overline{\alpha}(t)=\alpha(1-t)$ is the path $\alpha$ traversed in opposite direction and, for $i=1,2,\ldots,r$,}
$$\azul{\gamma_i(t)=\gamma\left(\frac{(i-1)\cdot t}{r-1}\right).}$$
Therefore
\begin{equation}\label{seqtial}
\mbox{\azul{$\text{TC}_{r,s}(f)=\text{sec}\left((1_{X^{r-s}}\times f^s)\circ e_r^\prime \right)$ \ \ and\ \ } $\text{HTC}_{r,s}(f)=\text{secat}\left((1_{X^{r-s}}\times f^s)\azul{{}\circ{}} e_r^\prime \right)$,}
\end{equation}
\azul{which explains the use of the name ``\emph{sequential} topological complexity'' as an alternative for ``\emph{higher} topological complexity''.}

\begin{remark}
\azul{As an abuse of notation, when using the ``sequential'' setting, we will keep writing $e^f_{r,s}$ 
for the map $(1_{X^{r-s}}\times f^s)\circ e'_r$ in~(\ref{seqtial}).}
\end{remark}

More generally, we can use other \azul{evaluation} maps to define $\text{TC}_{r,s}$ and $\text{HTC}_{r,s}$. \azul{For instance,} let~$G_r$ be any connected graph where $r$ ordered distinct vertices $v_1,\ldots,v_r$ have been selected, \azul{and consider} the evaluation map $e_{G_r}:X^{G_r}\to X^r$, $e_{G_r}(\gamma)=(\gamma(v_1),\ldots,\gamma(v_r))$. \azul{Then, as explained in~\cite[pages 2106--2107]{BGRT}, there are} commutative diagrams
\begin{eqnarray*}
\xymatrix{ X^{J_r} \ar[rr]^{ } \ar[rd]_{e_r^X} & & X^{G_r} \ar[ld]^{e_{G_r}} & \\
        &  X^r & &} & \xymatrix{ X^{G_r} \ar[rd]_{e_{G_r}}  \ar[rr]^{}  & & PX \ar[ld]^{e^\prime_r}& \\
        &  X^r & &}
\end{eqnarray*} \azul{which, together with~(\ref{seqtial}), yield}
\begin{align*}
\azul{\text{TC}_{r,s}(f)}&\azul{{}= \text{sec}\left((1_{X^{r-s}}\times f^s)\times e_{G_r} \right),}\\
\azul{\text{HTC}_{r,s}(f)}&\azul{{}= \text{secat}\left((1_{X^{r-s}}\times f^s)\times e_{G_r} \right).}
\end{align*}
\end{remark}

\begin{remark}\label{rem:fibration-fibration}
\noindent \begin{enumerate}
    \item By definition, the higher topological complexity $\text{TC}_{r,s}(1_X)$ of the identity map $1_X:X\to X$ coincides with the higher topological complexity $\text{TC}_r(X)$, i.e., $\text{TC}_{r,s}(1_X)=\text{TC}_r(X) \text{\azul{,} for any } s\in\{1,\ldots,r\}.$
    \item Note that $\text{HTC}_{r,s}(f)\leq \text{TC}_{r,s}(f)$ for any map $f$. Moreover, it is easy to see that $f$ is a fibration if and only if $e_{r,s}^f$ is a fibration \azul{(for instance, use Remark~\ref{remark:tc-path} in the proof of} \cite[Lemma 4.1]{pavesic2019}). So, we immediately obtain $\text{TC}_{r,s}(f)=\text{HTC}_{r,s}(f)$ for any fibration $f$.
\end{enumerate}
\end{remark}

\azul{The following result generalizes \cite[Proposition~3.3]{rudyak2010higher}.}

\begin{proposition}\label{prop:ineq} For any map \azul{$f\colon X\to Y$ and any $s=1,2,\ldots,r$,}
   $$\text{TC}_{r,s}(f)\leq\min\{\text{TC}_{r+1,s}(f),\text{TC}_{r+1,s+1}(f)\}.$$
\end{proposition}
\begin{proof}
Define $\mu_{r,s}:X^{J_{r+1}}\to X^{J_r}$ as the map which forgets the $(r+1-s)$-th path, for any $s\in\{1,\ldots,r\}$. \blue{Explicitly}, the $(r+1)$-tuple $\gamma=(\gamma_1,\ldots,\gamma_{r-s},\gamma_{r+1-s},\gamma_{r+2-s},\ldots,\gamma_{r+1})$ of paths $\gamma_i$ in $X$ \blue{is sent under $\mu_{r,s}$} to the $r$-tuple $\gamma=(\gamma_1,\ldots,\gamma_{r-s},\gamma_{r+2-s},\ldots,\gamma_{r+1})$ of paths. Choose $a\in X$ and \azul{consider the subspace inclusion $\varphi_a\times 1_{Y^s}\colon X^{r-s}\times Y^s\hookrightarrow X^{r+1-s}\times Y^s$, where} $$\varphi_a:X^{r-s}\hookrightarrow X^{r-s+1},~\varphi_a(x_1,\ldots,x_{r-s})=(x_1,\ldots,x_{r-s},a).$$
\azul{If $r=s$, we think of $X^{r-s}$ as the single-point space $\{a\}$, and ignore it in any cartesian product.} Take \azul{an open} cover $U_1,\ldots,U_m$ of $X^{r+1-s}\times Y^s$ such that each $U_i$ has a local section \azul{$\sigma_i\colon U_i\to X^{J_{r+1}}$} of $e_{r+1,s}^f$ for $i=1,\ldots,m$\azul{, and put} $$V_i=\azul{U_i\cap \left(X^{r-s}\times Y^{s}\right).}$$ Then a local section $s_i:\azul{V_i}\to X^{J_{\azul{r}}}$ of $e_{r,s}^f$ \azul{is given by}
\[V_i\stackrel{}{\hookrightarrow} U_i\stackrel{\azul{\sigma_i}}{\longrightarrow} X^{J_{r+1}}\stackrel{\mu_{r,s}}{\longrightarrow} X^{J_r}.\] 
\azul{This yields $\text{TC}_{r,s}(f)\leq\text{TC}_{r+1,s}(f)$.}

\smallskip
On the other hand, \azul{choose an element} $b\in Y$ and \azul{consider the subspace inclusion $1_{X^{r-s}}\times i_b\colon X^{r-s}\times Y^s\hookrightarrow X^{r-s}\times Y^{s+1}$, where}
 $$i_b:Y^{s}\to Y^{s+1},~i_b(z)=(b,z).$$ \azul{As before, take an open cover} $U_1,\ldots,U_m$ of $X^{r-s}\times Y^{s+1}$ such that each $U_i$ has a local section \azul{$\sigma_i\colon U_i\to X^{J_{r+1}}$} of $e_{r+1,s+1}^f$ for $i=1,\ldots,m$\azul{, and put}
 $$
 V_i=\azul{U_i\cap\left(X^{r-s}\times Y^s\right).}
 $$ Then a local section $s_i:\azul{V_i}\to X^{\azul{J_{r}}}$ \azul{of $e_{r,s}^f$ is given by}
 \[
 V_i\hookrightarrow U_i\stackrel{\azul{\sigma_i}}{\longrightarrow} X^{J_{r+1}}\stackrel{\mu_{r,s}}{\longrightarrow} X^{J_r}.
 \] 
\blue{We thus get} $\text{TC}_{r,s}(f)\leq\text{TC}_{r+1,s+1}(f)$.
\end{proof}

\begin{proposition}\label{prop:lower-cat-sec}
For a map $f:X\to Y$, we have $$\text{TC}_{r,s}(f)\geq \left\{
  \begin{array}{ll}
     \max\{\text{sec}\hspace{.1mm}(f^s),\text{ sec}^{\azul{1_{X^{r-s}}\times f^s}}(e_r^X), \text{ cat}(X^{r-s-1}\times Y^s)\}, & \hbox{for $s<r$;}\\
     \max\{\text{sec}(f^r),\text{ sec}^{\azul{f^r}}(e_r^X), \text{ TC}_{r}(Y)\}, & \hbox{for $s=r$.}
    \end{array}
\right.$$
\end{proposition}
\begin{proof}
\azul{Item} (4) of Lemma~\ref{lem:proper-sec-secat}, \azul{yields}

\vspace{-7mm}
\begin{eqnarray*}
\text{TC}_{r,s}(f) &=& \text{sec}(X^{J_r}\stackrel{e_r^X}{-\!\!\!\longrightarrow} X^r\stackrel{1_{X^{r-s}}\times f^s}{-\!\!\!-\!\!\!-\!\!\!-\!\!\!-\!\!\!\longrightarrow} X^{r-s}\times Y^s)\\
&\geq&\max\{\text{sec}(1_{X^{r-s}}\times f^s),\text{sec}^{(1_{X^{r-s}}\times f^s)}(e_r^X)\}\\
&=& \max\{\text{sec}(f^s),\text{sec}^{(1_{X^{r-s}}\times f^s)}(e_r^X)\}\azul{,}
\end{eqnarray*}  
\azul{where the last equality comes from} item (8) of Lemma~\ref{lem:proper-sec-secat}.

\smallskip
\azul{For $s<r$,} consider the the canonical pullback
\begin{eqnarray*}
\xymatrix{ (i_a)^\ast(e_{r,s}^{f}) \ar[d]_{ } \ar[r]^{ } & X^{J_r} \ar[d]^{\azul{e_{r,s}^{f}}} & \\
       X^{r-s-1}\times Y^s \ar[r]_{i_a}  & X^{r-s}\times Y^s\azul{,}  &}
\end{eqnarray*} 
\azul{where} $i_a:X^{r-s-1}\times Y^s\azul{{}\hookrightarrow{}} X^{r-s}\times Y^s$ \azul{is the subspace} inclusion  given \blue{by} \azul{$i_a(x,y)=(a,x,y)$, for some fixed $a\in X$.} \azul{Since} $(i_a)^\ast(e_{r,s}^{f})$ is contractible, item\azul{s~(1) and~(6)} of Lemma~\ref{lem:proper-sec-secat} \azul{yield} $\text{TC}_{r,s}(f)\geq\text{cat}(X^{r-s-1}\times Y^s)$. On the other hand, \azul{for $s=r$,} \blue{the commutative} diagram
\begin{eqnarray*}
\xymatrix{ X^{J_r}\ar[dr]_{e^f_{r,r}} \ar[rr]^{f_{\#}}  & & Y^{J_r}  \ar[dl]^{e_r^Y} \\ & Y^r & } 
\end{eqnarray*} yields the inequality $\text{TC}_r(Y)\leq\text{TC}_{r,r}(f)$.
\end{proof}

\subsection{\azulito{Rudyak-Soumen higher TC as a generalization of Murillo-Wu's TC}}\label{relationtoRS}
The \emph{\azulito{quasi-strong LS category}} of a map $f\colon X\to(Y,B)$, \azulito{$\mathrm{qscat}(f)$, introduced by Rudyak and Soumen in}~\cite[Definition 2.7]{rudyak2022}, is the least integer $n$ such that $X$ can be covered by $n$ open subsets $\{U_i\}_{i=1}^{n}$ on each of which there is a homotopy $H_i:U_i\times [0,1]\to Y$ satisfying $\left(H_i\right)_0=f|_{U_i}$ and $\left(H_i\right)_1(U_i)\subset B$.

\smallskip
\red{For any commutative diagram \begin{eqnarray}\label{pararepetir}
\xymatrix{ X \ar[r]^{\varphi} \ar[d]_{f} & X' \ar[d]^{h} & \\
       Y  \ar[r]_{g} &  Z, &}
\end{eqnarray}
it is easy to see that
\begin{equation}\label{easytosee}
\text{qscat}(g:Y\to (Z,B))\leq\text{sec}(f)\cdot\text{qscat}(h:X'\to (Z,B)).
\end{equation}
\azulito{Furthermore, if $Z$ is path-connected, then}} \red{$$\text{qscat}(g:Y\to (Z,B))\leq\text{cat}(g:Y\to Z),$$ \azulito{with} equality whenever $B$ is contractible.}
\begin{proposition}\label{propo37}
\azulito{Assume~(\ref{pararepetir}) is a quasi pullback with}
\red{$h:X'\to Z$ a fibration \azulito{admiting} a section over \azulito{a subspace $B$ of $Z$. Then}
$\text{sec}(f)\leq \text{qscat}(g:Y\to (Z,B))$.}
\end{proposition}
\begin{proof}
\red{Let $\sigma:B\to X'$ be a section of $h$ and $H\colon U\times I\to Z$ be a homotopy with $H_1=g_{|\azulito{U}}$ and $H_0(U)\subset B$. \azulito{The outer square in the diagram}
\begin{eqnarray*}
\xymatrix{ U \ar[r]^{\sigma\circ H_0} \ar[d]_{j_0} & X' \ar[d]^{h} & \\
       U\times I\ar@{-->}[ur]_{G}  \ar[r]_{H} &  Z &}
\end{eqnarray*}
\azulito{commutes and, since}
$h$ is a fibration, there is a homotopy $G\colon U\times I\to X'$ \azulito{that renders the complete diagram commutative.} 
Then the commutative diagram
\begin{eqnarray*}
\xymatrix@C=3cm{ U \ar@/^10pt/[drr]^{\,\,G_1} \ar@{^{(}->}@/_10pt/[ddr]_{ }  &   &  &\\ 
& X \ar[r]^{\,\,\varphi} \ar[d]^{f} & X' \ar[d]^{h} & \\
       & Y   \ar[r]_{\,\, g} &  Z &}
\end{eqnarray*} and the quasi pullback hypothesis \azulito{yield} a section $s\colon U\to X$ of $f$.}
\end{proof}

\azulito{Taking into account~(\ref{easytosee}) we then get:}
\begin{corollary}\label{cor:qscat-1}
\red{\azulito{Under the conditions in Proposition~\ref{propo37}},
$\text{sec}(f)= \text{qscat}(g:Y\to (Z,B))$ \azulito{provided} $\text{qscat}(h:X'\to (Z,B))=1$.}
\end{corollary}

\red{Corollary \ref{cor:qscat-1} implies \cite[Proposition 9.18, pg. 261]{cornea2003lusternik} as shown in the next example. 
\begin{example}
\azulito{Assume~(\ref{pararepetir}) is a quasi pullback with} $Z$ path-connected and $h:X'\to Z$ a null-homotopic fibration.
\azulito{Fix $z_0\in Z$ and} set $B=\{z_0\}$. Then $\mathrm{qscat}(h:X'\to (Z,B))=\mathrm{cat}(h:X'\to Z)=1$, \azulito{so that} $\mathrm{sec}(f)=\mathrm{qscat}(g:Y\to (Z,B))=\mathrm{cat}(g:Y\to Z)$.
\end{example}
}

\azulito{We next introduce the two central characters in this subsection.}

\smallskip
{\bf (A)} A notion of higher topological complexity of a map \azulito{has been introduced} \azul{in}~\cite{rudyak2022} by Rudyak and Soumen \azul{as follows.} For $r\geq 2$ and a map $f:X\to Y$, the $r$-higher topological complexity of $f$ (\azul{\`a la} Rudyak-Soumen), \azul{which} we denote $\mathrm{TC}^{RS}_r(f)$, is \red{given by $$\mathrm{TC}^{RS}_r(f)=\mathrm{qscat}\left(f^r\colon X^r\to (Y^r,\Delta_r(Y))\right),$$ that is,} the least integer $n$ such that $X^r$ can be covered by $n$ open subsets $\{U_i\}_{i=1}^{n}$ on each of which there is a homotopy $H_i:U_i\times [0,1]\to Y^r$ satisfying $\left(H_i\right)_0=f^r|_{U_i}$ and $\left(H_i\right)_1(U_i)\subseteq\Delta_r(Y)$, where $$\azul{\Delta_r(Y)=\{(y,\ldots,y)\in Y^r\colon y\in Y\}}$$ is the diagonal. \azul{More generally,} for $2\leq s\leq r$, \azul{set} $\Delta_{r,s}(Y)=Y^{r-s}\times \Delta_s(Y)$, \azul{and} define the \textit{$(r,s)$-th quasi-strong \azul{higher} topological complexity} of $f$, denote\azul{d by} $\mathrm{qsTC}_{r,s}(f)$, as the least integer $n$ such that $X^r$ can be covered by $n$ open subsets $\{U_i\}_{i=1}^{n}$ on each of which there is a homotopy $H_i:U_i\times [0,1]\to Y^r$ satisfying $\left(H_i\right)_0=f^r|_{U_i}$ and $\left(H_i\right)_1(U_i)\subseteq\Delta_{r,s}(Y)$\red{, that is, $$\mathrm{qsTC}_{r,s}(f)=\mathrm{qscat}\left(\rule{0mm}{3.9mm}f^r\colon X^r\to (Y^r,\Delta_{r,s}(Y))\right).$$} Note that $\mathrm{qsTC}_{r,r}(f)=\mathrm{TC}^{RS}_r(f)$ \red{and $\mathrm{qsTC}_{r,s'}(f)\leq\mathrm{qsTC}_{r,s}(f)$ for any $2\leq s'\leq s\leq r$.} 

\medskip{\bf (B)} \red{\azulito{Here is} a natural generalization of Murillo and Wu's complexity \azulito{reviewed at the end of Section~\ref{secprelim}.} \azulito{For} a map $f:X\to Y$ and $1\leq s\leq r\geq 2$, consider the \azulito{diagram} 
\[
\xymatrix{
X^{J_r} \ar[d]_{e_r^X} &  \\
        X^r \ar[rr]_{1_{X^{r-s}\times f^{\azulito{s}}}\rule{.8cm}{0mm}}  &  & X^{r-s}\times Y^s.}
\] The $(r,s)$-higher topological complexity of $f$ (\`a la Murillo-Wu), which we denote \azulito{by} $\mathrm{TC}^{MW}_{r,s}(f)$, is given by $$\mathrm{TC}^{MW}_{r,s}(f)=\mathrm{secat}^{1_{X^{r-s}}\times f^s}\left(e_r^X\right),$$ i.e., the least integer $n$ such that $X^r$ can be covered by $n$ open sets $\{U_i\}_{i=1}^{n}$ on each of which there is a map $s_i:U_i\to X^{J_r}$ satisfying $\left(1_{X^{r-s}}\times f^s\right)\circ e_r^{X}\circ s_i\simeq \left(1_{X^{r-s}}\times f^s\right)_{\mid U_i}$. Note that $\mathrm{TC}^{MW}_{r,s}(f)$ coincides with the least integer $n$ such that $X^r$ can be covered by $n$ open sets $\{U_i\}_{i=1}^{n}$ on each of which there is a map $s_i:U_i\to X$ satisfying $$\left(1_{X^{r-s}}\times f^s\right)\circ \Delta_{r}^X\circ s_i\simeq \left(1_{X^{r-s}}\times f^s\right)_{\mid U_i},$$ where $\Delta_{r}^X:X\to X^r,~x\mapsto (x,\ldots,x)$. \azulito{Likewise,} the naive $(r,s)$-higher topological complexity of $f$ (\`a la Murillo-Wu), which we denote by $\mathrm{tc}^{MW}_{r,s}(f)$, is defined analogously, \azulito{now} requiring each of the maps $s_i:U_i\to PX$ to satisfy the stronger condition $\left(1_{X^{r-s}}\times f^s\right)\circ e_r^{X}\circ s_i= \left(1_{X^{r-s}}\times f^s\right)_{\mid U_i}$. In other words, $$\mathrm{tc}^{MW}_{r,s}(f)=\mathrm{sec}^{1_{X^{r-s}}\times f^s}(e_r^{X}).$$ Note that the inequality $\mathrm{TC}^{MW}_{r,s}(f)\leq \mathrm{tc}^{MW}_{r,s}(f)$ holds for any map $f$, while in fact $\mathrm{TC}^{MW}_{r,s}(f)=\mathrm{tc}^{MW}_{r,s}(f)$ when $f$ is a fibration. We will write $\mathrm{TC}^{MW}_{r}(f)=\mathrm{TC}^{MW}_{r,r}(f)$ and $\mathrm{tc}^{MW}_{r}(f)=\mathrm{tc}^{MW}_{r,r}(f)$. \azulito{Of course} $$\mathrm{TC}^{MW}_{2}(f)=\mathrm{TC}^{MW}(f),$$ \azulito{the} Murillo-Wu's complexity.} 

\medskip
\red{The following statement generalizes \cite[Theorem 3.4]{scott2022} \azulito{and solves on the positive the question raised in~\cite{rudyak2022} by Rudyak-Soumen regarding their inequality~(3.6).}}
\begin{proposition}\label{rs-wm-s}
 \red{For $r\geq 2$ and a map $f:X\to Y$, we have}
 \[\red{\mathrm{TC}^{RS}_{r}(f)=\mathrm{TC}^{MW}_{r}(f)=\mathrm{sec}_{f^r}(e_r^Y).}\]
\end{proposition}
\begin{proof}
 \red{For $U\subset X^r$ and $\sigma:U\to Y^{J_r}$ satisfying $e_r^Y\circ \sigma=(f^r)_{\mid U}$, consider the homotopy $H:U\azulito{{}\times[0,1]}\to Y^r$ given by $$H(x,t)=\left(\overline{\sigma_1(x)}\cdot\sigma_r(x)(t),\ldots,\overline{\sigma_{r-1}(x)}\cdot\sigma_r(x)(t),f(x_r)\right).$$ Here, for each $x=(x_1,\ldots,x_r)\in U$, $\sigma(x)=\left(\sigma_1(x),\ldots,\sigma_r(x)\right)$ is an ordered $r$-\azulito{multipath} in~$Y$ \azulito{---see~(\ref{evaluation-fibration})}. Recall that \azulito{$\overline{\alpha}(t)=\alpha(1-t)$ is the path $\alpha$ traversed in opposite direction, and that $\alpha\cdot\beta$ stands for the concatenation of $\alpha$ and $\beta$.} Note that $H_0=(f^r)_{\mid U}$ and $H_1(U)\subset\Delta_r(Y)$\azulito{. This yields} $\mathrm{sec}_{f^r}(e_r^Y)\geq \mathrm{TC}^{RS}_{r}(f)$.} 

\smallskip\red{We \azulito{next argue} the inequality $\mathrm{TC}^{RS}_{r}(f)\geq\mathrm{TC}^{MW}_{r}(f)$. For $U\subset X^r$ and a homotopy $H:U\times [0,1]\to Y^r$ satisfying $H_0=(f^r)_{\mid U}$ and $H_1(U)\subset\Delta_r(Y)$, \azulito{set} $$\alpha_j(x)(t):=p_j(H(x,t))$$ for each $j=1,\ldots,r$, $x\in U$ and $t\in [0,1]$, where $p_j:Y^r\to Y$ is the projection to the $j$-th coordinate. \azul{Then} the homotopy $G:U\times [0,1]\to Y^r$ given by $$G(x,t)=\left(\alpha_1(x)\cdot\overline{\alpha_1(x)}(t),\alpha_2(x)\cdot\overline{\alpha_1(x)}(t),\ldots,\alpha_r(x)\cdot\overline{\alpha_1(x)}(t)\right)$$ satisfies $\azulito{G}_0=(f^r)_{\mid U}$ and $\azulito{G}_1=f^r\circ\Delta_r^X\circ\pi_1$, where $\pi_1(x_1,\ldots,x_r)=x_1$. \azulito{This yields the asserted inequality.}}

\smallskip\red{\azulito{We complete the proof by showing} the inequality $\mathrm{TC}^{MW}_{r}(f)\geq\mathrm{sec}_{f^r}(e_r^Y)$. For $U\subset X^r$ and $\azulito{s}:U\to X$ satisfying $f^r\circ\Delta_r^X\circ s\simeq (f^r)_{\mid U}$, consider the commutative diagram
    \begin{eqnarray*}
\xymatrix{ X \ar[r]^{c_f} \ar[d]_{\Delta_r^X} & Y^{J_r} \ar[d]^{e_r^Y} & \\
       X^r \ar[r]_{f^r} &  Y^r, &}
\end{eqnarray*} where $c_f:X\to Y^{J_r}$ is given \azulito{so that} $c_f(x)=\overline{f(x)}$, the constant map. \azulito{Then} the map $\sigma:U\to Y^{J_r}$ given by $\sigma=c_f\circ s$ defines a \azulito{homotopy} lift of $(f^r)_{\mid U}$ through 
$e_r^Y$ \azulito{The result follows since $e^Y_r$ is a fibration}.}
\end{proof}

\subsection{\azulito{The $\TC_{r,s}$ input}}\label{input}
\azulito{We start by comparing the generalized Murillo-Rudyak-Soumen-Wu} \red{complexity $\mathrm{sec}_{f^r}(e_r^Y)$ \azulito{to} $\mathrm{HTC}_{r,r}(f)$.}

\begin{proposition}\label{s-hth}\red{
    For $r\geq 2$ and a map $f:X\to Y$, we have:
    \begin{enumerate}
        \item $\mathrm{sec}_{f^r}(e_r^Y)\leq\mathrm{HTC}_{r,r}(f)\leq\mathrm{TC}_{r,r}(f)$.
        \item If $f$ admits a section, then $\mathrm{sec}_{f^r}(e_r^Y)=\mathrm{HTC}_{r,r}(f)=\mathrm{TC}_{r,r}(f).$
    \end{enumerate}}
\end{proposition}
\begin{proof}
\red{(1) \azulito{Choose} $U\subset Y^r$ and $s:U\to X^{J_r}$ satisfying $\azulito{f^r}\circ e_r^X\circ s\simeq incl_{U}$, \azulito{and} consider $V=(f^r)^{-1}(U)\subset X^r$\azulito{. Then} the map $\sigma:V\to Y^{J_r}$ given by $\sigma=f_{\#}\circ s\circ (f^r)_{\mid V}$ defines a \azulito{homotopy} lift of $(f^r)_{\mid V}$ through $e_r^Y$. \azulito{This yields} the inequality $\mathrm{sec}_{f^r}(e_r^Y)\leq\mathrm{HTC}_{r,r}(f)$\azulito{, so the proof is complete in view of item (2) in Remark~\ref{rem:fibration-fibration}.}}

\smallskip
\red{(2) \azulito{It suffices to show the inequality} $\mathrm{TC}_{r,r}(f)\leq\mathrm{sec}_{f^r}(e_r^Y)$ \azulito{assuming} that $s:Y\to X$ is a section of $f$. Let $U$ be an open subset of $X^r$, $\sigma:U\to Y^{J_r}$ \azulito{be} a lifting of $(f^r)_{\mid U}$ through~$e_r^Y$, and consider $V=(s^r)^{-1}(U)\subset Y^r$. \azulito{Then} the map $\rho:V\to X^{J_r}$ given by $\rho=s_{\#}\circ \sigma\circ(s^r)_{\mid U}$ defines a local section of $e_{r,r}^f=(f^r)\circ e_r^X$\azulito{, which yields the asserted inequality.}}
\end{proof}

\red{Proposition\azulito{s}~\ref{rs-wm-s} \azulito{and}~\ref{s-hth} \azulito{immediately yield:}}
\begin{corollary}\label{cor:all-tc}
\red{For $r\geq 2$ and a map $f:X\to Y$ which admits a section, we have}
\[\red{\mathrm{TC}^{RS}_{r}(f)=\mathrm{TC}^{MW}_{r}(f)=\mathrm{sec}_{f^r}(e_r^Y)=\mathrm{HTC}_{r,r}(f)=\mathrm{TC}_{r,r}(f).}\]
\end{corollary}
 
\azulito{Next we establish general estimates involving our \azulito{$\TC_{r,s}(-)$} and Rudyak-Soumen's $\mathrm{qsTC}_{r,s}(-)$.}

\begin{proposition}\label{rudyak}
 For any $2\leq s\leq r$ and any commutative diagram \begin{eqnarray*}
\xymatrix{ X \ar[r]^{\azul{\varphi}} \ar[d]_{f} & \azul{W} \ar[d]^{h} & \\
       Y  \ar[r]_{g} &  Z &}
\end{eqnarray*} we have \begin{enumerate}
    \item $\azul{\mathrm{HTC}_{r,s}(f)\cdot\mathrm{secat}\left(f^{r-s}\right)\cdot\mathrm{qsTC}_{r,s}(h)}\geq \mathrm{qsTC}_{r,s}(g)$.  
    \item $\mathrm{HTC}_{r,s}(f)\cdot\mathrm{TC}^{RS}_{s}(h)\geq\mathrm{secat}(f^s)\cdot\mathrm{TC}^{RS}_{s}(h)\geq \mathrm{TC}^{RS}_{s}(g)$.
\end{enumerate} In particular, $\mathrm{TC}_{r,r}(f)\cdot\mathrm{TC}^{RS}_{r}(h)\geq\mathrm{HTC}_{r,r}(f)\cdot\mathrm{TC}^{RS}_{r}(h)\geq \mathrm{TC}^{RS}_{r}(g).$
\end{proposition}

\begin{proof}
\azul{For (1), consider open sets $U$, $A$ and $V$, and maps $\sigma$, $\rho$ and $H$ satisfying}
\begin{itemize}
            \item[(i)] $U\subset X^{r-s}\times Y^s$, $\sigma:U\to X^{J_r}$ with $e_{r,s}^f\circ \sigma\simeq incl_U$;
            \item[(ii)] $A\subset Y^r$, $\rho:A\to X^{r-s}\times Y^s$ with $\left(f^{r-s}\times 1_{Y^s}\right)\circ\rho\simeq incl_A$;
            \item[(iii)] $V\subset W^r$, $H:V\times [0,1]\to Z^r$ with $H_0=h^r|_V$ and $H_1(V)\subset \Delta_{r,s}(Z)$.
        \end{itemize} \azul{(In (ii) we are using the equality $\text{sec}(f^{r-s})=\text{sec}(f^{r-s}\times 1_{Y^s})$ coming from item (8) of Lemma~\ref{lem:proper-sec-secat}.)} \azul{Consider also} the diagram
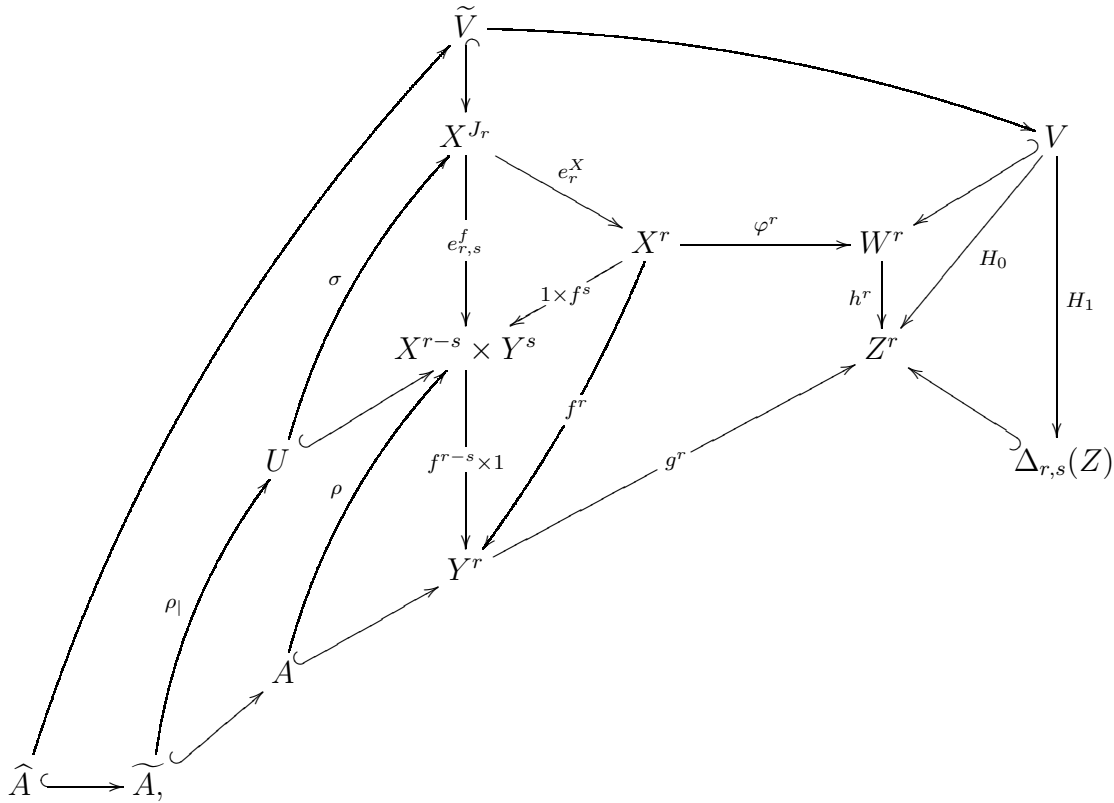
\begin{eqnarray*}
\xymatrix@C=1cm{ & & &\widetilde{V}\ar@/^10pt/[rrrrd]^{}\ar@{^{(}->}[d]_{ } &   &  & &\\ 
& & & X^{J_r}\ar[dd]|-{\rule{0mm}{4mm}{e_{r,s}^f}_{ \rule{0mm}{2.5mm}}}\ar[dr]^{e_r^X}&  &  &  & V\ar@{_{(}->}[dl]_{ }\ar[ddl]^{H_0}\ar[ddd]^{H_1} \\
& & & & X^r \ar[dl]|-{\azul{1}\times f^s\rule{0mm}{2.6mm}}\ar[rr]^{\varphi^r}\ar@/^5pt/[dddl]|-{\rule{1mm}{0mm}\rule{0mm}{2.7mm}f^r_{\rule{0mm}{2mm}}} &  & W^r\ar[d]_{h^r} & \\
& &  & X^{r-s}\times Y^s\ar[dd]|-{\rule{0mm}{3mm}f^{r-s}\times \azul{1}} &  &  &  Z^r& \\
& & U\ \ar@{^{(}->}[ur]\ar@/^10pt/[ruuu]^{\sigma} & &  &  &  & \hspace{2mm} \Delta_{r,s}(Z)\ar@{^{(}->}[lu]_{ }\\
& & & Y^r\ar[rrruu]|-{\rule{1mm}{0mm}g^r\rule{.5mm}{0mm}}&  &  & &  \\
& & A\ar@{^{(}->}[ur]\ar@/^10pt/[ruuu]^{\rho} & &  &  & & \\
\azul{\widehat{A}\hspace{1mm}} \ar@{^{(}->}[r]_{ }\ar@/^20pt/[rrruuuuuuu]^{ } 
& \widetilde{A\azul{,}}\hspace{1mm}\ar@{^{(}->}[ur]_{ } \ar@/^10pt/[ruuu]^{\rho_| }& & &  &  & & }
\end{eqnarray*} where $\widetilde{A}=\rho^{-1}(U)$, $\widetilde{V}=\left(\varphi^r\circ e_r^X\right)^{-1}(V)$ and $\azul{\widehat{A}}=\left(\sigma\circ\rho_|\right)^{-1}(\widetilde{V})$. \azul{All regions of the diagram are strictly commutative, except for the three homotopy commutative triangles involving the homotopies in (i), (ii) and (iii).} Note that the sets $\widetilde{A}$, $\widetilde{V}$ \azul{and} $\azul{\widehat{A}}$ can be empty but, when $\azul{\widehat{A}}\neq\varnothing$, we can take the homotopy $G:\azul{\widehat{A}}\times [0,1]\to Z^r$ given by \[G(y,t)=H\left(\varphi^r\circ e_r^X\circ\sigma\circ\rho(y),t\right).\] Then $G_0\simeq g^r|_{\azul{\widehat{A}}}$ and $G_1(\azul{\widehat{A}})\subset \Delta_{r,s}(Z)$. \azul{The asserted inequality~(1) then follows by observing that, as the sets $U$, $A$ and $V$ vary over suitable coverings, the resulting sets $\widehat{A}$ cover $Y^r$.}  

\medskip \azul{Regarding (2), the inequality $\text{HTC}_{r,s}(f)\geq\text{secat}(f^s)$ is obvious, so we focus on the second inequality of (2). Consider open sets $A$ and $V$ and maps $\rho$ and $H$ satisfying:} \begin{itemize}
            \item[(iv)] $A\subset Y^s$, $\rho:A\to X^{s}$ with $f^{s}\circ\rho\simeq incl_A$;
            \item[(v)] $V\subset W^s$, $H:V\times [0,1]\to Z^s$ with $H_0=h^s|_V$ and $H_1(V)\subset \Delta_{s}(Z)$.
        \end{itemize} \azul{Consider also} the diagram
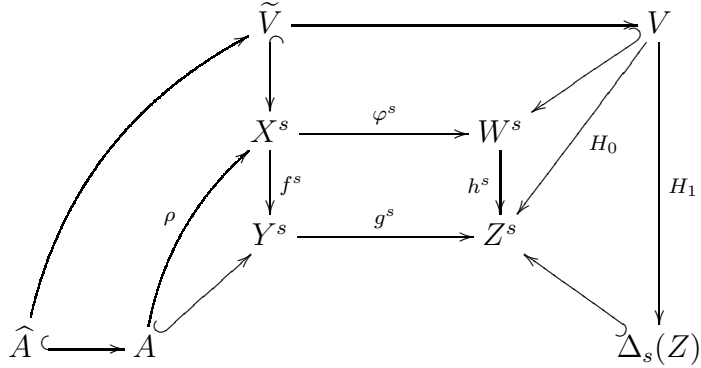
\begin{eqnarray*}
\xymatrix@C=1cm{  & &\widetilde{V}\ar@{^{(}->}[d]_{ }\ar[rrr]_{ } &   &  & V\ar@{_{(}->}[dl]_{ }\ar[ddl]^{H_0}\ar[ddd]^{H_1}\\ 
 & & X^s\ar[d]^{f^s}\ar[rr]^{\varphi^s}&  & W^s\ar[d]_{h^s} &   \\
  & & Y^s \ar[rr]^{g^s} &  & Z^s & \\ \azul{\widehat{A}\hspace{1mm}}\ar@{^{(}->}[r]_{ }\ar@/^20pt/[rruuu]^{ }
 & \azul{A}\ar@{^{(}->}[ur]\ar@/^10pt/[ruu]^{\azul{\rho}} & &  &  & \Delta_{s}(Z)\ar@{^{(}->}[lu]_{ }}
\end{eqnarray*} where $\widetilde{V}=\left(\varphi^s\right)^{-1}(V)$ and $\azul{\widehat{A}}=\azul{\rho}^{-1}(\widetilde{V})$. \azul{All regions of the diagram are strictly commutative, except for the two homotopy commutative triangles involving the homotopies in (iv) and (v).} Note that the sets $\widetilde{V}$ \azul{and $\widehat{A}$} can be empty but, when $\azul{\widehat{A}}\neq\varnothing$, we can take the homotopy $G:\azul{\widehat{A}}\times [0,1]\to Z^s$ given by \[G(y,t)=H\left(\varphi^s\circ\azul{\rho}(y),t\right).\] Then $G_0\simeq g^s|_{\azul{\widehat{A}}}$ and $G_1(\hat{A})\subset \Delta_{s}(Z)$. \azul{The second inequality in (2) now follows by observing that, as the sets $A$ and $V$ vary over suitable coverings, the resulting sets $\widehat{A}$ cover $Y^s$.}
\end{proof}

\subsection{Products}

The following result was proved in \cite[Proposition~22, pg. 84]{schwarz1966}. It will be used in the proof of Proposition~\ref{tc-seq-product}. Here we agree that a normal space is, by definition, required to be Hausdorff.

\begin{lemma}\label{sec-product}
Let $f\times f^\prime:X\times X^\prime\to Y\times Y^\prime$ be the product of two maps $f:X\to Y$ and $f^\prime:X^\prime\to Y^\prime$. If $Y\times Y^\prime$ is normal, then $$\text{sec}(f\times f^\prime)\leq\text{sec}(f)+\text{sec}(f^\prime)-1.$$
\end{lemma}

In \cite[Proposition 3.11]{BGRT} the authors obtained the subadditivity of $\text{TC}_r$ under suitable topological hypothesis. The corresponding property for higher topological complexity of maps is given next.

\begin{proposition}\label{tc-seq-product}
Let $f:X\to Y$ and $f^\prime:X^\prime\to Y^\prime$ \azul{be} two maps. If \azul{the cartesian product} $(X\times X^\prime)^{r-s}\times (Y\times Y^\prime)^{s}$ is normal, then \[\text{TC}_{r,s}(f\times f^\prime)\leq \text{TC}_{r,s}(f)+\text{TC}_{r,s}(f^\prime)-1.\]
\end{proposition}
\begin{proof}
The proof proceeds by analogy with \azul{the proof of} \cite[Proposition~3.11]{BGRT}. Indeed, consider \blue{the  commutative diagram \azul{with horizontal homeomorphisms}}
\begin{eqnarray*}
\xymatrix{ (X\times X^\prime)^{J_r} \ar[d]_{e_{r,s}^{f\times f^\prime} } \ar[r]^{\varphi} & X^{J_r}\times {X^\prime}^{J_r} \ar[d]^{\,e_{r,s}^{f}\times e_{r,s}^{f^\prime}} & \\
       (X\times X^\prime)^{r-s}\times (Y\times Y^\prime)^s \ar[r]_{\psi\hspace{6mm}}  & \left(X^{r-s}\times Y^s\right)\times \left({X^\prime}^{r-s}\times {Y^\prime}^s\right)\azul{.}  &}
\end{eqnarray*}
\azul{Here $\varphi\left(\gamma:J_r\to X\times X^\prime\right) := \left(\rule{0mm}{3.8mm}p_X\circ\gamma:J_r\to X,\;p_{X^\prime}\circ\gamma:J_r\to X^\prime\right)$, while}  
\begin{align*}
\azul{\psi\left(\rule{0mm}{3.8mm}(x_1,x_1^\prime),\ldots,\right.}&\azul{\left((x_{r-s},x_{r-s}^\prime),(y_1,y^\prime_1),\ldots,(y_s,y^\prime_s)\right)} \\
&\azul{:= ((x_1,\ldots,x_{r-s},y_1,\ldots,y_s), (x_1^\prime,\ldots,x_{r-s}^\prime,y^\prime_1\ldots,y^\prime_s)),}
\end{align*}
where $x_i\in X$, $y_i\in Y$, $x^\prime_i\in X^\prime$ and $y^\prime_i\in Y^\prime$, \azul{and where $p_X$ and $p_{X'}$ are the obvious projections.} \blue{The} desired conclusion \blue{then} follows from Lemma~\ref{sec-product}.
\end{proof}

\subsection{Effect of pre- and post-composition}

We study the effect \blue{on the higher topological complexity of maps under} pre- and post-composition.

\begin{lemma}\label{sec-diagram}
Consider \blue{the \azul{commutative} diagram} 
\begin{eqnarray*}
\xymatrix{ X^\prime \ar[d]^{f^\prime} & X \ar[r]^{\varphi} \ar[d]_{f} & X^\prime \ar[d]^{f^\prime} & \\
     Y^\prime \ar[r]_{\xi} & Y  \ar[r]_{\psi} &  Y^\prime\azul{.} &}
\end{eqnarray*} 
\begin{enumerate}
    \item If $\psi\circ \xi\simeq 1_{Y^\prime}$ then $\text{secat}(f)\geq\text{secat}(f^\prime)$.
    \item If $\psi\circ \xi=1_{Y^\prime}$ then $\text{sec}(f)\geq\text{sec}(f^\prime)$  \azul{(and, of course, $\text{secat}(f)\geq\text{secat}(f^\prime)).$}
\end{enumerate}
\end{lemma}
\begin{proof}
Suppose $U\subset Y$ and take $V=\azul{\xi}^{-1}(U)\subset Y^\prime$. Note that a map $\sigma:U\to X$ yields a map $\delta=\left(V\stackrel{\xi}{\to} U\stackrel{\sigma}{\to} X\stackrel{\varphi}{\to} X^\prime\right).$ 
If $\psi\circ \xi=1_{Y^\prime}$ \azul{($\psi\circ \xi\simeq1_{Y^\prime}$, respectively)} and $f\circ\sigma=incl_U$ ($f\circ\sigma\simeq incl_U$, respectively), then $f^\prime\circ\delta=incl_V$ ($f^\prime\circ\delta\simeq incl_V$\azul{,} respectively).
\end{proof}

\begin{proposition}\label{tc-section-map}
Consider the diagram of maps $W\stackrel{h}{\to} X\stackrel{f}{\to} Y\stackrel{g}{\to} Z$. 
 \begin{itemize}
     \item[($a$)] If $f$ admits a section \azul{(homotopy section, respectively),} then \[\text{TC}_{r,s}(f\circ h)\leq \text{TC}_{r,s}(h) \hspace{2mm} \azul{\left(\hspace{.3mm}\rule{0mm}{4mm}\text{HTC}_{r,s}(f\circ h)\leq \text{HTC}_{\azul{r,s}}(h), \text{ respectively}\right)\!,} \text{ for any } s\leq r.\]
     \item[($b$)] If $f$ admits a homotopy section, then 
\begin{align}
\text{HTC}_{r,s}(g)&\leq \text{HTC}_{r,s}(g\circ f) \text{\azul{,} for any } \azul{s\leq r;}\label{versionH}\\
\text{TC}_{r,s}(g)&\leq \text{TC}_{r,s}(g\circ f)  \text{\azul{,} for any } \azul{s<r.}\label{versionS}
\end{align}
 \end{itemize}
  In particular, if $f$ admits a section \azul{and $s\leq r$,} we get \[\text{TC}_{r}(Y) \leq \text{HTC}_{r,s}(f)\leq\text{TC}_{r,s}(f)\leq\text{TC}_{r}(X)\azul{.} \] 
\end{proposition}
\begin{proof} 
\azul{We use the sequential setting.} Item $(a)$ follows Lemma~\ref{sec-diagram} \azul{applied} to the \azul{commutative} \blue{diagram}
\begin{eqnarray*}
\xymatrix{ PW \ar[d]^{e^{f\circ h}_{r,s}} & PW \ar[r]^{\,\,1} \ar[d]_{e^{h}_{r,s}} & PW \ar[d]^{e^{f\circ h}_{r,s}} & \\
     W^{r-s}\times Y^s \ar[r]_{\azul{1}\times \xi^s} & W^{r-s}\times X^{s}  \ar[r]_{\azul{1}\times f^{s}} &  W^{r-s}\times Y^s\azul{,} &}
\end{eqnarray*} 
\azul{where $\xi\colon Y\to X$ is either a section or a homotopy section of $f$. On the other hand, for item ($b$), assume \emph{only} that $\xi:Y\to X$ is a homotopy section to $f$, and consider the commutative diagram}
\begin{eqnarray*}
\xymatrix{ PY \ar[d]^{e^{g}_{r,s}} & PX \ar[r]^{\,\,f_{\#}} \ar[d]_{e^{g\circ f}_{r,s}} & PY \ar[d]^{e^{g}_{r,s}} & \\
     Y^{r-s}\times Z^s \ar[r]_{\xi^{r-s}\times \azul{1}\hspace{1mm}} & X^{r-s}\times Z^{s}  \ar[r]_{f^{r-s}\times \azul{1}\hspace{1mm}} &  Y^{r-s}\times Z^s\blue{.} &}
\end{eqnarray*} 
\azul{Since~(\ref{versionH}) follows also from Lemma~\ref{sec-diagram}, we will focus on (\ref{versionS}) assuming $s<r$ (in addition to $f\circ\xi\simeq 1_Y$).} 

\smallskip
Choose a homotopy $H:f\circ \xi\simeq 1_Y$ \azul{ and suppose we are given an open set } $U\subset X^{r-s}\times Z^{s}$ \azul{admitting} a \azul{local} section \azul{$\sigma:U\to PX$} of $e^{g\circ f}_{r,s}$. \azul{It is then elementary to check that a local section $\delta$ of $e^g_{r,s}$  on $V:=(\xi^{r-s}\times 1_{Z^s})^{-1}(U)$ is given, in terms of concatenation of paths, by} the formula
\begin{align*}
 \delta(v) =& \left(\azul{\overline{H(y_1,-)}}\cdot (f\circ\sigma((\xi(y_1),\ldots,\xi(y_{r-s}),z_{r-z+1},\ldots,z_r))\mid_{1}) \cdot H(y_2,-)\right)\cdot\\ 
 &  \left(\azul{\overline{H(y_2,-)}}\cdot (f\circ\sigma((\xi(y_1),\ldots,\xi(y_{r-s}),z_{r-z+1},\ldots,z_r))\mid_{2}) \cdot H(y_3,-)\right)\cdot\cdots \cdot\\ 
 &  \left(\azul{\overline{H(y_{r-s-1},-)}}\cdot (f\circ\sigma((\xi(y_1),\ldots,\xi(y_{r-s}),z_{r-z+1},\ldots,z_r))\mid_{r-s-1}) \cdot H(y_{r-s},-)\right)\cdot\\
 &  \left(\azul{\overline{H(y_{r-s},-)}}\cdot (f\circ\sigma((\xi(y_1),\ldots,\xi(y_{r-s}),z_{r-z+1},\ldots,z_r))\mid_{r-s})\right) \cdot\\
 &  \left(\rule{0mm}{4mm}f\circ\sigma((\xi(y_1),\ldots,\xi(y_{r-s}),z_{r-z+1},\ldots,z_r))\mid_{r-s+1} \right)\cdot\cdots\cdot\\
 &  \left(\rule{0mm}{4mm}f\circ\sigma((\xi(y_1),\ldots,\xi(y_{r-s}),z_{r-z+1},\ldots,z_r))\mid_{r-1} \right),
 \end{align*} for any $v=(y_1,\ldots,y_{r-s},z_{r-z+1},\ldots,z_r)\in V$. \azul{Here, $\overline{\tau}$ is the path~$\tau$ traversed backwards (see Remark~\ref{remark:tc-path}). Furthermore,}
$\sigma((\xi(y_1),\ldots,\xi(y_{r-s}),z_{r-z+1},\ldots,z_r))\mid_{j}$ \azul{stands for} the restriction of $\sigma((\xi(y_1),\ldots,\xi(y_{r-s}),z_{r-z+1},\ldots,z_r))$ to the segment $$\left[\dfrac{j-1}{\azul{r-1}},\dfrac{j}{\azul{r-1}}\right],$$ i.e., $\sigma((\xi(y_1),\ldots,\xi(y_{r-s}),z_{r-z+1},\ldots,z_r))\mid_j(t)$ is given by the formula  \[\sigma((\xi(y_1),\ldots,\xi(y_{r-s}),z_{r-z+1},\ldots,z_r))\left(\azul{\dfrac{t+j-1}{r-1}}\right), \quad t\in [0,1],\]  for $j=1,\ldots,r-1$. \azul{Since the sets $V$ cover $Y^{r-s}\times Z^s$ as the sets $U$ cover $X^{r-s}\times Z^s$, we get the desired inequality} $\text{TC}_{r,s}(g)\leq \text{TC}_{r,s}(g\circ f)$. 
\end{proof}

\begin{remark}\label{remark:tc-up-section}
\azul{It is highly illuminating to take a look back to item ($b$) of Propostion~\ref{tc-section-map} and its proof. For starters, it should be stressed that~(\ref{versionS}) involves the \emph{strong} form of the higher TC, even though the hypothesis on~$f$ has a homotopy nature. Such a phenomenon works because of the additional hypothesis $s<r$. Indeed, the first four lines in the definition of $\delta(v)$ allow us to incorporate the homotopy $H$ into a pullback-type construction (involving the homotopy section $\xi$) of the strict section $\delta$ out of the strict section $\sigma$. Of course, such a trick would not be need if $f$ \emph{had} a strict section, as then~(\ref{versionS}) would be true for any $s\leq r$ (using the ``same'' argument that proves~(\ref{versionH})). But then, it is more striking to remark that~(\ref{versionH}) and~(\ref{versionS}) actually have stronger forms when $s=r$, as spelled out next.}
\end{remark}

\begin{proposition}\label{prop:tcrr-map-with-section}
Let $f:X\to Y$ and $g:Y\to Z$ be maps.
\begin{enumerate}
    \item \azul{Independently of whether $f$ admits a (homotopy) section, we have} $$\text{TC}_{r,r}(g)\leq \text{TC}_{r,r}(g\circ f) \text{ \ \ and \ \ }\text{HTC}_{r,r}(g)\leq \text{HTC}_{r,r}(g\circ f).$$ In particular, $\text{TC}_r(Y)\leq \text{HTC}_{r,r}(f)\leq \text{TC}_{r,r}(f)$.  
    \item If $f$ admits a section \azul{(homotopy section, respectively)}, then 
    \[\text{TC}_{r,r}(g)=\text{TC}_{r,r}(g\circ f) \ \ \ \ \azul{(\text{HTC}_{r,r}(g)=\text{HTC}_{r,r}(g\circ f), \text{respectively})}.\] 
    In particular  $\text{TC}_r(Y)=\text{TC}_{r,r}(f)$ \ \ \ (\azul{$\text{TC}_r(Y)=\text{HTC}_{r,r}(f)$, respectively}).
\end{enumerate}
\end{proposition}
\begin{proof}
\azul{Working again in the sequential context, item (1) follows immediately by applying} Lemma~\ref{sec-diagram} to the \blue{diagram}
\begin{eqnarray*}
\xymatrix{ PY \ar[d]^{e^{g}_{r,r}} & PX \ar[r]^{\,\,f_{\#}} \ar[d]_{e^{g\circ f}_{r,r}} & PY \ar[d]^{e^{g}_{r,r}} & \\
     Z^r \ar[r]_{\,\,1_{Z^r}} & Z^{r}  \ar[r]_{\,\,1_{Z^r}} &  Z^r\blue{.} &}
\end{eqnarray*} 
Moreover, if $f$ admits a section $\sigma:Y\to X$, \azul{then} $\text{TC}_{r,r}(g\circ f)\leq \text{TC}_{r,r}(g\circ f\circ\sigma)=\text{TC}_{r,r}(g)$, \azul{so in fact} $\text{TC}_{r,r}(g)=\text{TC}_{r,r}(g\circ f)$. \azul{Likewise, if $\sigma:Y\to X$ is} a homotopy section of $f$, \azul{then} $\text{HTC}_{r,r}(g\circ f)\leq \text{HTC}_{r,r}(g\circ f\circ\sigma)=\text{HTC}_{r,r}(g)$, \blue{so in fact} $\text{HTC}_{r,r}(g)=\text{HTC}_{r,r}(g\circ f)$. 
\end{proof}

The \azul{facts we have discussed in this subsection} have \azul{a number of} interesting corollaries. First, we deduce the following important invariance property, which states that the complexity of the map is not altered by a deformation retraction of the domain.

\begin{corollary}
If $\rho:X^\prime\to X$ is a deformation retraction, then for \azul{any} $f:X\to Y$ \azul{and any $s\leq r$} we have \[\text{TC}_{r,s}(f)=\text{TC}_{r,s}(f\circ \rho) \text{ \ \ and \ \ } \text{HTC}_{r,s}(f)=\text{HTC}_{r,s}(f\circ \rho).\]
\end{corollary}
\begin{proof}
Let $i:X\hookrightarrow X^\prime$ be the inclusion map, so that $\rho\circ i=1_X$ and $i\circ \rho\simeq 1_{X^\prime}$. Because $\rho$ admits a section, the case $s=r$ follows from Proposition~\ref{prop:tcrr-map-with-section}. \azul{So, we assume} $s<r$. Item $\azul{(b)}$ of Proposition~\ref{tc-section-map} implies $\text{TC}_{r,s}(f)\leq \text{TC}_{r,s}(f\circ \rho)$ \azul{on the nose, as well as} $\text{TC}_{r,s}(f\circ \rho)\leq \text{TC}_{r,s}(f)$, \azul{since} $(f\circ \rho)\circ i=f$. Similarly, we get the equality $\text{HTC}_{r,s}(f)=\text{HTC}_{r,s}(f\circ \rho)$.
\end{proof}

The following \azul{fact (written in the sequential setting) is analogous to} \cite[Lemma 4.6]{cesarsectional}.

\begin{lemma}\label{lemma:pullback}
If $f:X\to Y$ is a fibration and $f^\prime:Y\to Y^\prime$ is a map, then \blue{we have the \azul{quasi} pullback diagram} \begin{eqnarray*}
\xymatrix{ PX \ar[r]^{\,\,f_{\#}} \ar[d]_{e^{f^\prime\circ f}_{r,r-1}} & PY \ar[d]^{e^{f^\prime}_{r,r-1}} & \\
       X\times {Y^\prime}^{r-1}  \ar[r]_{\,\, f\times 1_{{Y^\prime}^{r-1}}} &  Y\times {Y^\prime}^{r-1}\blue{.} &}
\end{eqnarray*}
\end{lemma}
\begin{proof}
 \azul{Choose maps $\beta$ and $\alpha$ that render the commutative diagram} 
 \begin{eqnarray*}
\xymatrix@C=3cm{ Z \ar@/^10pt/[drr]^{\,\,\beta} \ar@/_10pt/[ddr]_{\alpha} \ar@{-->}[dr]_{H} &   &  &\\
& PX \ar[r]^{\,\,f_{\#}} \ar[d]^{e^{f^\prime\circ f}_{r,r-1}} & PY \ar[d]^{e^{f^\prime}_{r,r-1}} & \\
       & X\times {Y^\prime}^{r-1}   \ar[r]_{\,\, f\times 1_{{Y^\prime}^{r-1}}} &  Y\times {Y^\prime}^{r-1}\azul{,} &}
\end{eqnarray*}
 \azul{when the dashed map $H$ is ignored. We need to construct a map $H$ that still fits in the commutative diagram.} 
 
 \smallskip \azul{Consider} the commutative diagram
\begin{eqnarray*}
\xymatrix{ Z \ar[r]^{\,\,p_{1}\circ\alpha} \ar[d]_{i_0} & X \ar[d]^{f} \\
       Z\times I  \ar[r]_{\azul{\widehat{\beta}}} &  Y\azul{,}}
\end{eqnarray*} where $p_1$ is the projection onto the first coordinate and $\azul{\widehat{\beta}}:Z\times I\to Y$ is given by $\azul{\widehat{\beta}}(z,t)=\beta(z)(t)$. Because $f$ is a fibration, there exists $G:Z\times I\to X$ \azul{rendering the commutative diagram}
\begin{eqnarray*}
\xymatrix{ Z \ar[r]^{\,\,p_{1}\circ\alpha} \ar[d]_{i_0} & X \ar[d]^{f} \\
       Z\times I\ar[ru]_{G}  \ar[r]_{\hat{\beta}} &  Y\azul{.}}
\end{eqnarray*} \azul{It is elementary to check that the map} $H:Z\to PX$ given by $H(z)(t)=G(z,t)$ \azul{has the required property.}
\end{proof}

\azul{Just as Proposition~\ref{prop:tcrr-map-with-section} specializes to $s=r$, the next result specializes to $s=r-1$ providing a} generalization of~\cite[Proposition 4.7]{cesarsectional}. \azul{The proof follows directly from item~(1) of Lemma~\ref{lem:proper-sec-secat} and Lemma~\ref{lemma:pullback}.}

\begin{corollary}\label{prop:tcr-r-1-fibration}
If $f:X\to Y$ is a fibration, then $\text{TC}_{r,r-1}(f^\prime\circ f)\leq \text{TC}_{r,r-1}(f^\prime)$ for any map $f^\prime:Y\to Y^\prime$. In particular, $\text{TC}_{r,r-1}(f)\leq \text{TC}_r(Y)$.
\end{corollary}

\azul{In turn,} Proposition~\ref{tc-section-map} \azul{and Corollary}~\ref{prop:tcr-r-1-fibration} \azul{can be combined} to deduce the following
important property, which states that the $(r,r-1)$-complexity of a fibration \azul{admiting} a homotopy section depends only of the complexity of its co-domain.

\begin{corollary}\label{cor:fibration-section-r-r-1}
If $f:X\to Y$ is a fibration that admits a homotopy section, then \[\text{TC}_{r,r-1}(f)=\text{TC}_{r}(Y).\]
\end{corollary}

\begin{example}
For the projection $p_X:X\times F\to X$ we have $\text{TC}_{r,r-1}(p_X)=\text{TC}_{r}(X).$
\end{example}

\red{Item (2) of Proposition~\ref{prop:tcrr-map-with-section} \azulito{together with Corollaries}~\ref{cor:all-tc} \azulito{and~\ref{cor:fibration-section-r-r-1} yield} the following \azulito{omnibus} statement, which \azulito{comprises the fact that, for large values of $s$, $\TC_{r,s}$} \azulito{unifies} previous \azul{notions of topological complexity}.}

\begin{corollary}\label{cor:unification}
   \red{If $f:X\to Y$ is a fibration that admits a homotopy section, then for any $r\geq 2$, we have }
   \[\red{\mathrm{TC}^{RS}_{r}(f)=\mathrm{TC}^{MW}_{r}(f)=\mathrm{sec}_{f^r}(e_r^Y)=\mathrm{HTC}_{r,r}(f)=\mathrm{TC}_{r,r}(f)=\text{TC}_{r,r-1}(f)=\text{TC}_{r}(Y).}\] 
\end{corollary}

\subsection{Homotopy invariance}
Recall that two maps $f:X\to Y$ and $f^\prime:X^\prime\to Y$ are said to be \textit{fibre homotopy
equivalent} (or FHE-equivalent) if there are commutative diagrams of the form 
\begin{eqnarray*}
\xymatrix{ X \ar[rr]^{\psi} \ar[rd]_{f} & & X^{\prime} \ar[ld]^{f^\prime} & \\
        &  Y & &} & \xymatrix{ X^{\prime} \ar[rd]_{f^\prime}  \ar[rr]^{\varphi}  & & X \ar[ld]^{f}& \\
        &  Y & &}
\end{eqnarray*}
and the maps $\varphi\circ\psi$ and $\psi\circ\varphi$ are homotopic to the respective identity map by fibre preserving
homotopies. 

In \cite[Corollary 3.9]{pavesic2019} the author \blue{proved} the FHE-invariance of $\text{TC}(f)$.
\azul{A generalization of the} corresponding property for the higher case $\text{TC}_{r,s}(f)$ is given next.

\begin{proposition}
Given $f:X\to Y$ and $f^\prime:X^\prime\to Y$\blue{,} assume that there exist fibrewise maps $\psi:X\to X^\prime$ and $\varphi:X^\prime\to X$ that \blue{are} homotopy inverses \blue{of each} other. Then \[\text{TC}_{r,s}(f)=\text{TC}_{r,s}(f^\prime) \text{ \ \ and \ \ } \text{HTC}_{r,s}(f)=\text{HTC}_{r,s}(f^\prime)\azul{,}\]
\azul{for any $s\leq r$.} In particular, the \azul{$(r,s)$-}higher topological complexity is a FHE-invariant.
\end{proposition}
\begin{proof}
By Proposition~\ref{tc-section-map} and Proposition~\ref{prop:tcrr-map-with-section} we have
\begin{eqnarray*}
\text{TC}_{r,s}(f) &=& \text{TC}_{r,s}(f^\prime\circ\psi)\geq \text{TC}_{r,s}(f^\prime)= \text{TC}_{r,s}(f\circ\varphi)\geq \text{TC}_{r,s}(f),
\end{eqnarray*} \azul{so} $\text{TC}_{r,s}(f)=\text{TC}_{r,s}(f^\prime).$ Similarly, we get the equality $\text{HTC}_{r,s}(f)=\text{HTC}_{r,s}(f^\prime)$.
\end{proof}

On the other hand, from item (3) of Lemma~\ref{lem:proper-sec-secat} we \azul{see} that the \azul{homotopy higher topological} complexity is a homotopy invariant\azul{:}

\begin{proposition}
If $f\simeq g$ then $\text{HTC}_{r,s}(f)=\text{HTC}_{r,s}(g)$\azul{, for any $s\leq r$.}
\end{proposition}

\subsection{\azul{Upper bounds}}

\blue{In} \cite[Theorem 3.17]{pavesic2019}, \azul{as corrected in version 2 of the Arxiv version, Pave\v{s}i\'{c}} presents an upper bound  of $\text{TC}(f)$ for any map $f$. \azul{We next generalize} such a \azul{fact by giving} an upper estimate \azul{for} the \azul{$(2s,s)$}-higher \azul{topological} complexity \azul{of} any map~$f$.

\begin{proposition}\label{general-upper-estimate}
Let $f:X\to Y$ be a map with $X$ path-connected and $X^s\times Y^s$ normal. We have
\begin{equation}\label{general-estimate}
    \text{TC}_{2s,s}(f)\leq\text{cat}(X^s)+\text{cat}(X^s)\cdot\text{sec}(f^s)-1.
\end{equation}
\end{proposition}
\begin{proof}
\blue{Fix} $x_0\in X$ \azul{and let $U$ be an open subset of $X^s$} so that there exists a homotopy $H:U\times \azul{[0,1]}\to X^s$ \azul{from} the inclusion \azul{$U\hookrightarrow X^s$ to} the constant map
to $(x_0,\ldots,x_0)\in X^s$. \azul{Assume also that $s:V\to X^s$ is a local section of $f^s$ on an open subset $V$ of  $Y^s$. The} map $K:(V\cap s^{-1}(U))\times \azul{[0,1]}\to X^s$ given by $K(v,t)=H(s(v),t)$ is a homotopy from the restriction of $s$ to the constant map to $(x_0,\ldots,x_0)$. \azul{Then, in terms of concatenation of paths, the} formula
\begin{align*}
\delta(u,v) &= \left[(p_1\circ H(u,-))\cdot (p_2\circ\azul{\overline{H(u,-)}}\hspace{.4mm})\right]\cdot 
\left[(p_2\circ H(u,-))\cdot (p_3\circ\azul{\overline{H(u,-)}\hspace{.4mm}})\right]\cdot\cdots\cdot\\ & \hspace{5mm} 
\left[(p_{s-1}\circ H(u,-))\cdot (p_s\circ\azul{\overline{H(u,-)}}\hspace{.4mm})\right]\cdot 
\left[(p_s\circ H(u,-))\cdot (p_1\circ\azul{\overline{K(v,-)}}\hspace{.4mm})\right]\cdot\\ & \hspace{5mm} 
\left[(p_1\circ K(v,-))\cdot (p_2\circ\azul{\overline{K(v,-)}}\hspace{.4mm})\right]\cdot\cdots\cdot
\left[(p_{s-1}\circ K(v,-))\cdot (p_s\circ\azul{\overline{K(v,-)}}\hspace{.4mm})\right]
\end{align*}
defines a local section to $e_{2s,s}^f$ over $U\times (V\cap s^{-1}(U))$. \azul{Here, $p_i\colon X^s\to X$ stands for the $i$-th projection and, as in Remark~\ref{remark:tc-path}, $\overline{\tau}$ stands for the path $\tau$ traversed in the opposite direction.}  \azul{The conclusion then follows from} item (7) of Lemma~\ref{lem:proper-sec-secat}.
\end{proof}

\blue{The} estimate (\ref{general-estimate}) is \blue{sharp under special conditions:}

\begin{corollary}\label{cor:tc2s-s-contractible}
Let $f:X\to Y$ be a map \azul{with $X$ contractible and} $Y^s$ normal. Then
\[\text{TC}_{2s,s}(f)=\text{sec}(f^s).\]
\end{corollary}
\begin{proof}
\azul{Use Propositions~\ref{prop:lower-cat-sec} and Proposition~\ref{general-upper-estimate}.
}\end{proof}

\azul{Relative sectional numbers $\text{sec}^{-}(-)$ can also be used to draw} estimates. Specifically, \blue{item (4) of} Lemma~\ref{lem:proper-sec-secat} \azul{yields:}

\begin{proposition}
For any map $f:X\to Y$, we have \[\text{TC}_{r,s}(f)\leq\text{sec}(f^s)\cdot\text{sec}^{1_{X^{r-s}}\times f^s}(e^X_{r}).\] 
\end{proposition}

\begin{corollary}\label{cor:TCr-s-sec}
Let $f:X\to Y$ be a \blue{map.}
\begin{enumerate}
    \item If $f$ admits a section, $\text{TC}_{r,s}(f)=\text{sec}^{1_{X^{r-s}}\times f^s}(e^X_{r}).$
    \item If $X$ is contractible, $\text{TC}_{r,s}(f)=\text{sec}(f^s)$.
\end{enumerate}
\end{corollary}
\begin{proof}
Recall the lower \azul{estimate} $\max\{\text{sec}(f^s),\text{sec}^{1_{X^{r-s}}\times f^s}(e^X_{r})\}\leq \text{TC}_{r,s}(f)$ \azul{in} Proposition~\ref{prop:lower-cat-sec} and the upper estimate $\text{sec}^{1_{X^{r-s}}\times f^s}(e^X_{r})\leq\text{TC}_r(X)$ \azul{coming from} Lemma~\ref{lem:sec-composite}. 
\end{proof}

\azul{Note that item} (2) of Corollary~\ref{cor:TCr-s-sec} generalizes Corollary~\ref{cor:tc2s-s-contractible}.

\begin{example}
Let $f:X\to Y$ be a map. If $f$ admits a section, then $$\azul{\text{TC}_r(Y)}=\text{TC}_{r,r}(f)=\text{sec}^{f^r}(e^X_{\azul{r}}).$$ The \azul{former} equality follows from \blue{item \azul{(2)} of} Proposition~\ref{prop:tcrr-map-with-section}.
\end{example}

\subsection{Higher complexity of a fibration}
We \azul{now} \blue{obtain \azul{new} estimates} \azul{for} $\text{TC}_{r,s}(f)$ when $f$ is a fibration. Firstly, we restate the definition of $\text{TC}_{r,s}(f)$ (for $s<r$) in more geometric terms. Recall that a \textit{deformation} of $U\subset Z$ in $Z$ to a subset $V\subset Z$ is a map $H:U\times I\to Z$ such that $H(u,0)=u$ \blue{and} $H(u,1)\in V$\blue{,} for all $u\in U$.

\begin{proposition}
Let $f:X\to Y$ be a fibration, and let $U\subset X^{r-s}\times Y^s$ \azul{with} $s<r$. \azul{The} following statements are equivalent:
\begin{enumerate}
    \item \azul{There is} a local section $\sigma:U\to X^{J_r}$ \azul{for $e^f_{r,s}$.}
    \item $U$ can be deformed in $X^{r-s}\times Y^s$ to the subset
    \begin{equation}\label{thesubset}\blue{\Delta_f=\{(x,\ldots,x,f(x),\ldots,f(x))\in X^{r-s}\times Y^s:~x\in X\}.}\end{equation}
\end{enumerate}
\end{proposition}
\begin{proof}
$(1)\Longrightarrow (2)$. \azul{The homotopy} $H:U\times \azul{[0,1]}\to X^{r-s}\times Y^s$ \azul{sending $(u,t)$ to}
\[\left(\rule{0mm}{4mm}{\sigma}(\azul{u)((1-t)_1}),\ldots,{\sigma}(\azul{u)((1-t)_{r-s}}),f({\sigma}(\azul{u)((1-t)_{r-s+1}})),\ldots,\azul{f}({\sigma}(\azul{u,(1-t)_r}))\right)\] deforms $U$ in $X^{r-s}\times Y^s$ to $\Delta_f$\azul{. Here, for $x\in[0,1]$, the notation $x_i$ stands for the copy of $x$ lying in the $i$-th wedge summand of $[0,1]$ in $J_r$.}

$(2)\Longrightarrow (1)$. Let $p_i$ denote the projection to the $i$\azul{-th} factor and \azul{choose a lifting function} $\Gamma:\azul{E_f}\to \azul{PX}$ of the fibration $f$ \azul{as in~(\ref{lifun}).} Given a deformation $H:U\times \azul{[0,1]}\to X^{r-s}\times Y^s$ \azul{of $U$ to $\Delta_f$,} we define a section $\sigma:U\to X^{J_r}$ \azul{for $e^f_{r,s}$} by
\[
\sigma(\azul{u}) = \left(p_1\azul{{}\circ{}}\azul{\overline{H(\azul{u},-)}},\ldots,p_{r-s}\azul{{}\circ{}}\azul{\overline{H(\azul{u},-)}},\Gamma(\ast,p_{r-s+1}\azul{{}\circ{}}\azul{\overline{H(\azul{u},-)}\hspace{.4mm})},\ldots,\Gamma(\ast,p_{r}\azul{{}\circ{}}\azul{\overline{H(\azul{u},-)}}\hspace{.4mm})\right)\hspace{-.9mm}, 
\] where $\ast=p_1(H(\azul{u},1))=\cdots=p_{r-s}(H(\azul{u},1))$ (here we use \azul{the hypothesis} $s<r$) \azul{ and $\overline{\tau}$ stands for the path $\tau$ traversed in the opposite direction.} 
\end{proof}

\begin{corollary}
If $f:X\to Y$ is a fibration and $s<r$, then $\text{TC}_{r,s}(f)$ equals the minimal number of elements of a covering of $X^{r-s}\times Y^s$ by open sets that can be deformed in $X^{r-s}\times Y^s$ to the set \blue{in~(\ref{thesubset}).}
\end{corollary}

\begin{proposition}\label{estimates-tcrs}
If $f$ is a fibration then\blue{:}
\begin{enumerate}
    \item $\text{cat}(X^{r-s-1}\times Y^s)\leq\text{TC}_{r,s}(f)\leq\text{cat}(X^{r-s}\times Y^s)$\blue{,} for $s<r$.
    \item $\text{cat}(Y^{r-1})\leq\text{TC}_{r,r-1}(f)\leq\min\{\text{TC}_r(Y),\text{cat}(X\times Y^{r-1})\}.$
    \item $\max\{\text{sec}(f^r),\text{TC}_r(Y)\}\leq\text{TC}_{r,r}(f)\leq\text{cat}(Y^r)$. 
\end{enumerate}
\end{proposition}
\begin{proof}
   Because $f$ is a fibration, \azul{the map} $e_{r,s}^f:X^{J_r}\to X^{r-s}\times Y^s$ is a fibration \azul{too} (see item~(2) of Remark~\ref{rem:fibration-fibration}). Then, by item (5) of Lemma~\ref{lem:proper-sec-secat}, we obtain $\text{TC}_{r,s}(f)=\text{sec}(e_{r,s}^f)\leq \text{cat}\left(X^{r-s}\times Y^s\right)$ for any $1\leq s\leq r$.
  
In addition, from \azul{Corollary}~\ref{prop:tcr-r-1-fibration}, we get that $\text{TC}_{r,r-1}(f)\leq \text{TC}_r(Y)$. \azul{All the} lower estimates follow from Proposition~\ref{prop:lower-cat-sec}.   
\end{proof}

\blue{T}he upper estimate $\azul{\TC_{r,s}(f)\leq{}}\text{cat}(X^{r-s}\times Y^s)$ for \azul{$s\leq r$} in Proposition~\ref{estimates-tcrs} is \blue{sharp under special conditions} (see Corollary~\ref{cor:tc-equal-cat} below). However, \blue{there is room for improvement\azul{, as it can be seen from Corollary~\ref{h-space} below and, in particular, from Remark~\ref{evenso} in the final section of the paper, where the upper estimate in item~(2)} of Proposition~\ref{estimates-tcrs} becomes sharp due to the $\TC_r$ term.}

\begin{corollary}\label{cor:tc-equal-cat}
Let $f:X\to Y$ be a fibration \azul{and} assume that $X$ is contractible. Then
\[\text{TC}_{r,s}(f)=\text{cat}(Y^s)=\text{sec}(f^s)\blue{,} \text{ for any } s\leq r.\]
\end{corollary}

\begin{example}
If $f:\widetilde{X}\to X$ is \blue{the} universal covering of a spherical space $X$, then $\text{TC}_{r,s}(f)=\text{cat}(X^s)=\text{sec}(f^s)$ for $s\leq r$.
\end{example}

\azul{It is well known that, if} $Y$ is a topological \blue{group, or} more generally an $H$-\blue{space,} then the $r$-higher \blue{topological} complexity of $Y$ coincides \blue{with} $\text{cat}(Y^{r-1})$. \blue{As a consequence:}

\begin{corollary}\label{h-space}
Let $f:X\to Y$ be a \azul{fibration over an $H$-space} $Y$. Then \[\text{TC}_{r,r-1}(f)=\text{cat}(Y^{r-1})=\text{TC}_r(Y).\]
\end{corollary}

\begin{example}
If $f:X\to Y$ is a fibration \azul{with} a section\azul{, then ~(\ref{versionS}), item~(4) in Proposition~\ref{lem:proper-sec-secat} and item~(2) in Proposition~\ref{estimates-tcrs} yield} $\text{TC}_{r,r-1}(f)=\text{TC}_r(Y)=\text{sec}^{\azul{1_X\times{}}f^{r-1}}(e^X_{r})$.
\end{example}

\begin{remark}\label{monorefinado}
Item (2) of Proposition~\ref{estimates-tcrs} \blue{together with} 
Proposition\blue{s}~\azul{\ref{prop:ineq} \blue{and~\ref{prop:lower-cat-sec}}} yield $$\azul{\TC_r(Y)\leq \text{TC}_{r,r}(f)\leq\text{TC}_{r+1,r}(f)\leq \text{TC}_{r+1}(Y)\blue{,}}$$ for any fibration $f:X\to Y$.
\end{remark}

\subsection{Cohomological lower bound}

\azul{{\v{S}}varc's} cohomological lower bound for the sectional category \azul{of a map, a tool widely used \blue{in computations, arises as follows.}} \azul{A multiplicative cohomology theory} \blue{$h^\ast$} on the homotopy category of \blue{pairs} of spaces \azul{comes equipped with a} relative cohomology product \[\cup:h^\ast(X,A)\otimes h^\ast(X,B)\to h^\ast(X,A\cup B)\] \azul{whenever} $A,B\subset X$ are excisive. In our case, $A$ and $B$ \azul{will be} open sets. \azul{On the other hand, consider the \textit{index of nilpotence}
$$\text{nil}(S)=\min\{n:~\text{every product of $n$ elements \blue{in $S$ vanishes}}\}$$
defined for a subset $S$ of a ring $R$.}

\begin{lemma}[{\azul{\cite[Theorem 4 on page 73]{schwarz1966}}}]\label{lem:cohomo-secat}
For any map $f:X\to Y$, we \azul{have}
\[\text{nil}\left(\text{Ker}(f^\ast:h^\ast(Y)\to h^\ast(X))\right)\leq\text{secat}(f).\]
\end{lemma}

\azul{In our context:}
\begin{proposition}\label{theorem:cohomological-estimate}
For every map $f:X\to Y$ and for every multiplicative cohomology theory $h^\ast$, we have
\[\text{nil}\left(\text{Ker}((\Delta_{r-s},{}^sf)^\ast:h^\ast(X^{r-s}\times Y^s)\to h^\ast(X))\right)\leq\text{HTC}_{r,s}(f),\] where $(\Delta_{r-s},{}^sf):X\to X^{r-s}\times Y^s$ \blue{is} given by  $(\Delta_{r-s},{}^sf)=(1_{X^{r-s}}\times f^s)\circ \Delta_r$\blue{,} with $\Delta_r:X\to X^r,~x\mapsto (x,\ldots,x)$\azul{, the} diagonal map.
\end{proposition}

\begin{proof}
\azul{In the sequential context, c}onsider \blue{the commutative diagram}
\begin{eqnarray*}
\xymatrix{ PX\ar[dr]_{e^f_{r,s}}  & & X \ar[ll]_{c} \ar[dl]^{(\Delta_{r-s},{}^sf)} \\ & X^{r-s}\times Y^s\blue{,} & } 
\end{eqnarray*} where $c:X\to PX$ is \blue{the} homotopy \blue{equivalence given} by $c(x)=x$\blue{,} the constant path \blue{at~$x$.} \azul{The result follows from Lemma~\ref{lem:cohomo-secat} as} $\text{nil}\left(\text{Ker}((e^f_{r,s})^\ast)\right)=\text{nil}\left(\text{Ker}((\Delta_{r-s},{}^sf)^\ast)\right)$. 
\end{proof}

Although \blue{Proposition~\ref{theorem:cohomological-estimate}} is formulated in general terms, we will mostly \blue{consider cases where \azul{the k\"unneth formula}} $h^\ast(X^{r-s}\times Y^s)\blue{{}\cong{}} {h^\ast(X)}^{\otimes(r-s)}\otimes {h^\ast(Y)}^{\otimes s}$. \azul{In such cases,} the action of $(\Delta_{r-s},{}^sf)^\ast$ on tensors \azul{of factors} $\alpha_1,\ldots,\alpha_{r-s}\in h^\ast(X)$ \azul{and} $\beta_1,\ldots,\beta_s\in h^\ast(Y)$ \azul{is given by the product:}
\[(\Delta_{r-s},{}^sf)^\ast(\alpha_1\otimes\cdots\otimes\alpha_{r-s}\otimes \beta_1\otimes\cdots\otimes\beta_s)=\alpha_1\cdots\alpha_{r-s}\cdot f^\ast(\beta_1)\cdots f^\ast(\beta_s).\]
\blue{In concrete cases \azul{(e.g.~those worked out in Section~\ref{sec:example} below)}} we do not attempt to compute the entire kernel of the homomorphism
$(\Delta_{r-s},{}^sf)^\ast$\azul{,} but we rather look for specific elements in the kernel and try to find long
non-trivial products.

\section{Examples}\label{sec:example}

\subsection{The complexity $\text{TC}_{r,s}(\azul{p_n}:S^n\to\mathbb{R}P^n)$}\label{covering-map}
Recall \azul{from \cite[Corollary~3.12]{BGRT}} the higher \azul{topological }complexity of the \azul{$n$-th} sphere \azul{$S^n$, $n\geq1$:}
\begin{equation}\label{tcresferas}
\text{TC}_r(S^n)=\begin{cases}
r, & \hbox{ if $n$ is odd,}\\
r+1, & \hbox{ if $n$ is even}.
\end{cases}
\end{equation}

Consider the \blue{usual double} covering map $\azul{p_n}:S^n\to\mathbb{R}P^n$. \azul{Since $\text{cat}(S^n)=2$ and $\text{cat}(\mathbb{R}P^n)=n+1$,} Proposition~\ref{estimates-tcrs} \azul{and the subadditivity of the Lusternik-Schnirelmann category yield the upper estimate}
\begin{equation}\label{estimacionpordebajo}
\text{TC}_{r,s}(\azul{p_n})\leq sn+r-s+1 \text{\blue{, \ } for any }  s\leq r.
\end{equation}
\azul{For a lower estimate, start by noticing} that $\azul{p_n}^\ast:H^\ast(\mathbb{R}P^n;\mathbb{Z}_2)\to H^\ast(S^n;\mathbb{Z}_2)$ is trivial \azul{in positive dimensions.} Set $\iota\in H^n(S^n;\mathbb{Z}_2)$\blue{, the} fundamental class of the sphere~$S^n$\blue{, and let} $\alpha\in H^1(\mathbb{R}P^n;\mathbb{Z}_2)$ \blue{be} the generator of the cohomology ring $H^\ast(\mathbb{R}P^n;\mathbb{Z}_2)=\mathbb{Z}_2[\alpha]/\left(\alpha^{n+1} \right)$. Set $v_i=q^\ast_i\alpha\in H^{\azul{1}}((S^n)^{r-s}\times (\mathbb{R}P^n)^s;\mathbb{Z}_2)$\azul{,} where $q_i:\azul{(S^n)}^{r-s}\times \azul{(\mathbb{R}P^n)}^s\to \azul{\mathbb{R}P^n}$ \azul{is the} projection onto the $i$\azul{-}th factor \azul{($r-s+1\leq i\leq r$).} Note that $\blue{0\neq{}}v_i\in \text{Ker} (\Delta_{r-s},{}^s\azul{p_n})^\ast$. \azul{In fact, the} product $\prod_{i=r-s+1}^{r} v_i^n$ \blue{does not vanish} so that \begin{equation}\label{lower-Y}
   \text{TC}_{r,s}(\azul{p_n})\geq sn+1.
\end{equation}  
In particular, (\ref{estimacionpordebajo}) and (\ref{lower-Y}) yield
\begin{equation}\label{losdearriba}
\azul{\text{TC}_{r,r}(\azul{p_n}) = rn+1\text{ \ \ and \ \ }\text{TC}_{r,r-1}(\azul{p_n}) = (r-1)n+\epsilon_{r-1},}
\end{equation}
where $\epsilon_{r-1}\in\{1,2\}$\azul{. So we next assume in addition} $r-s\geq 2$. \azul{For $i=1,2,\ldots, r-s$, set} $u_i=\azul{q}^\ast_i \iota\in H^{\azul{n}}((S^n)^{r-s}\times (\mathbb{R}P^n)^s;\mathbb{Z}_2)$ and $w_i=u_i+u_{r-s}\azul{{}\in \text{Ker} (\Delta_{r-s},{}^s\azul{p_n})^\ast,}$ where $\azul{q_i}:\azul{(S^n)}^{r-s}\times \azul{(\mathbb{R}P^n)}^s\to \azul{S^n}$ is the projection onto the $i$\azul{-}th factor.
  
\azul{Then} 
\[
    \prod_{i=1}^{r-s-1}w_i\ \cdot \!\prod_{i=r-s+1}^{r} \!v_i^n=\sum_{j=1}^{r-s}u_1\cdots \widehat{u_j}\cdots u_{r-s}\ \cdot\!\prod_{i=r-s+1}^{r} v_i^n\neq 0
\text{\blue{,}}\] \blue{so that} $\text{TC}_{r,s}(\azul{p_n})\geq sn+r-s$, \blue{which is a linear improvement over} (\ref{lower-Y}). \blue{Taking into account~(\ref{estimacionpordebajo}), we \azul{then} see that~(\ref{losdearriba}) extends to
\begin{equation}\label{losdearribaextendidos}
\text{TC}_{r,s}(\azul{p_n}) = sn+\epsilon_s,\text{ for any } \azul{s\leq r,}
\end{equation} where $\epsilon_s\in\{r-s,r-s+1\}$ \azul{and, in fact, $\epsilon_r=1$.}
}

\begin{remark}\label{evenso}
\azul{Assume $n\in\{1,3,7\}$, so that $\mathbb{R}P^n$ has the structure of an $H$-space.} Corollary~\ref{h-space} \azul{then yields}
\begin{equation}\label{tcrr-1p}
\text{TC}_{r,r-1}(\azul{p_n})=\text{cat}\azul{((\mathbb{R}P^n)^{r-1})}=\text{TC}_r(\mathbb{R}P^n)=(r-1)n+1\azul{.}
\end{equation}
Note here \azul{that} $\text{cat}(S^n\times(\mathbb{R}P^n)^{r-1})=(r-1)n+2>\text{TC}_{r,r-1}(\azul{p_n})$\azul{, which is relevant for the discussion in the paragraph following the proof of Proposition~\ref{estimates-tcrs}. In addition, we note that the following constructions have been done in~\cite[Section~5]{aguilar2021}:}
\begin{itemize}
\item \azul{For $n\in\{1,3,7\}$, an explicit partition of $S^n\times \mathbb{R}P^n$ into $n+1$ subsets, each admitting a section for $e^{\azul{p_n}}_{2,1}:PS^n\to S^n\times \mathbb{R}P^n$, thus realizing~(\ref{tcrr-1p}) when $r=2$.}
\item \azul{For general $n$, an explicit partition of $S^n\times \mathbb{R}P^n$ into $n+2$ subsets, each admitting a section for $e^{\azul{p_n}}_{2,1}:PS^n\to S^n\times \mathbb{R}P^n$, thus realizing the estimate $\epsilon_{r-1}\leq2$ in~(\ref{losdearribaextendidos}).}
\end{itemize}
\end{remark}
       
Similarly, for the standard quotient map $q:S^{2n+1}\to\mathbb{C}P^n$\azul{,} we obtain the estimate $$sn+r-s\leq\text{TC}_{r,s}(q)\leq sn+r-s+1,$$ for any $s\leq r$.

\subsection{Fibrations over spheres}\label{overspheres}

\azul{For a fibration $f:X\to S^n$,~(\ref{tcresferas}) and} item (2) of Proposition~\ref{estimates-tcrs} \azul{yield}
$$
    r = \text{cat}((S^n)^{r-1}) \leq \text{TC}_{r,r-1}(f)\leq \text{TC}_r(S^n)\leq r+1.
$$
\azul{In particular, for} $n$ odd, we \azul{actually} have $\text{TC}_{r,r-1}(f)=\azul{\TC_r(S^n)}=r$. 
    On the other hand,\azul{~(\ref{tcresferas}),} Proposition~\ref{prop:lower-cat-sec} and item (3) of Proposition~\ref{estimates-tcrs} \azul{yield}
$$r\leq\text{TC}_{r}(S^n)\leq\text{TC}_{r,r}(f)\leq\text{cat}((S^n)^r)=r+1.$$
\azul{In particular, if $n$ is even, we get in fact} $\text{TC}_{r,r}(f)=\azul{\TC_r(S^n)}=r+1$.

\bibliographystyle{plain}

\begin{thebibliography}{10}

\bibitem{aguilar2021} N.~\azul{Cadavid-Aguilar,} J.~González, B.~Gutiérrez, and C.~A.~I.~Zapata. {\it `Effectual topological complexity'}. Journal of Topology and Analysis (2021), 1-18.

\bibitem{BGRT} I.~Basabe, J.~Gonz\'{a}lez, Y.B.~Rudyak and D.~Tamaki. 
               {\it `Higher topological complexity and its symmetrization'}. 
               Algebr.~\& Geom.~Topol. 
               {\bf 14} (2014), no. 4, 2103--2124.
               
\bibitem{berstein1961} I. Berstein, T. Ganea. {\it `The category of a map and a cohomology class'}. Fund. Math. 50
(1961/1962), 265–279. 

\bibitem{cornea2003lusternik} O.~Cornea, G.~Lupton, J.~Oprea and D.~Tanr{\'e}. {\it `Lusternik-Schnirelmann Category'}. Mathematical Surveys and Monographs, 103 (American Mathematical Society, Providence, RI, 2003).

\bibitem{jmcal} J.M.~Garc\'{\i}a-Calcines.
               {\it `A note on covers defining relative and sectional categories'}.
               Topology Appl. 
               {\bf 265} (2019), paper number 106810.

\bibitem{gonzalez2019} J.~Gonz\'alez, M.~Grant and L.~Vandembroucq. {\it `Hopf invariants for sectional category with applications to topological robotics'}. The Quarterly Journal of Mathematics 70.4 (2019): 1209-1252. 

\bibitem{hatcheralgebraic} A. Hatcher. {\it `Algebraic topology'} (2001).

\bibitem{MP} L.E.~Kavraki and S.M.~LaValle. {\it `Motion planning'}. Chapter 5 of Handbook of Robotics. Springer-Verlag, Berlin Heidelberg (2008). 

\bibitem{murillo2019} A.~Murillo and J.~Wu. {\it `Topological complexity of the work map'}. Journal of Topology and Analysis 13.01 (2021): 219-238.

\bibitem{pavesic2017} P.~Pave{\v{s}}i{\'c}. {\it `Complexity of the forward kinematic map'}. Mechanism and Machine Theory,
117 (2017) 230–243.

\bibitem{pavesic2018} P.~Pave{\v{s}}i{\'c}. {\it `A Topologist’s view of kinematic maps and manipulation complexity'}. Contemp.
Math. 702 (2018) 61–84.

\bibitem{pavesic2019} P.~Pave{\v{s}}i{\'c}. {\it `Topological complexity of a map'}. Homology, Homotopy and Applications. International Press of Boston. {\bf 21} (2019), no. 2, 107--130. ArXiv preprint arXiv:1809.09021 (2019).

\bibitem{rudyak2010higher} Y. Rudyak. {\it `On higher analogs of topological complexity'}. Topology and its Applications, Elsevier. {\bf 157} (2010), no. 5, 916--920. Erratum in Topology Appl. (157) 2010 1118.

\bibitem{rudyak2022} Y. Rudyak, and S. Soumen. {\it `Relative LS categories and higher topological complexities of maps'}. Topology and its Applications 322 (2022): 108317.

\bibitem{schwarz1966} A.~S.~Schwarz. {\it `The genus of fiber space'}. Amer.~Math. Sci, Transl. \azul{{S}eries 2,} {\bf 55}: \azul{{E}leven papers on topology and algebra} (1966), 49--140.

\bibitem{scott2022}  J.~Scott. {\it `On the topological complexity of maps'}. Topology and its Applications 314 (2022), Paper No. 108094, 25 pp.

\bibitem{cesarmilnor} C.~A.~I.~Zapata. {\it `Cross-sections of Milnor fibrations and Motion planning'}. ArXiv preprint arXiv:1910.00157 (2019).

\bibitem{zapata2022} C.~A.~I.~Zapata. {\it `Espaços de configurações no problema de planificação de movimento simultâneo livre de colisões'}. Ph.D thesis, Universidade de São Paulo, 2022.

\bibitem{cesarsectional} C.~A.~I.~Zapata and J.~González. {\it `Sectional category and The Fixed Point Property'}. Topol. Methods Nonlinear Anal. Volume 56, Number 2 (2020), 559--578.

\end{thebibliography}

\end{document}